\documentclass[English]{amsart}%

\usepackage[utf8]{inputenc}
\usepackage{amsmath} %
\usepackage{amsthm}
\usepackage{amsfonts}
\usepackage{amssymb}
\usepackage{fullpage}
\usepackage{mathtools}
\usepackage{stmaryrd}
\usepackage{mathrsfs}
\usepackage[bookmarks,hidelinks,pdfencoding=unicode]{hyperref}

\usepackage[capitalize]{cleveref}
\usepackage{listings}
\usepackage{enumerate}
\usepackage{bbm}
\usepackage{accents}
\usepackage{quiver}
\usepackage{caption}
\usepackage{multirow}
\usepackage{adjustbox}
\usepackage{censor}
\newtheorem{thm}{Theorem}[section]

\newtheorem{prop}[thm]{Proposition}
\newtheorem{lem}[thm]{Lemma}
\newtheorem{cor}[thm]{Corollary}
\newtheorem{fact}[thm]{Fact}

\newtheorem{quest}[thm]{Question}

\theoremstyle{definition}
\newtheorem{defn}[thm]{Definition}

\theoremstyle{remark}

\newenvironment{claim}[1]{\par\noindent\emph{Claim.}\space#1}{}
\newenvironment{claimproof}[1]{\par\noindent\emph{Proof of claim.}\space#1}{\hfill $\qed_{\text{claim}}$}

\newcommand{\Qb}{\mathbb{Q}}

\newcommand{\Lc}{\mathcal{L}}

\newcommand{\Tc}{\mathcal{T}}
\newcommand{\Hc}{\mathcal{H}}
\newcommand{\Cc}{\mathcal{C}}

\makeatletter \DeclareRobustCommand{\cset}{\@ifstar\star@cset\normal@cset}
\newcommand{\star@cset}[1]{{\left\llbracket#1\right\rrbracket}}
\newcommand{\normal@cset}[2][]{{\mathopen{#1\llbracket}#2\mathclose{#1\rrbracket}}}
\makeatother

\newcommand{\tleq}{\trianglelefteq}

\newcommand{\xbar}{\bar{x}}
\newcommand{\ybar}{\bar{y}}

\newcommand{\abar}{\bar{a}}
\newcommand{\dbar}{\bar{d}}
\newcommand{\bbar}{\bar{b}}
\newcommand{\cbar}{\bar{c}}

\newcommand{\mbar}{\bar{m}}
\newcommand{\nbar}{\bar{n}}
\newcommand{\tbar}{\bar{t}}
\newcommand{\ibar}{\bar{\imath}}

\newcommand{\sbar}{\bar{s}}

\newcommand{\gbar}{\bar{g}}
\newcommand{\hbarr}{\bar{h}}

\newcommand{\IZF}{\ifmmode\mathsf{IZF}\else$\mathsf{IZF}$\fi}
\newcommand{\CZF}{\ifmmode\mathsf{CZF}\else$\mathsf{CZF}$\fi}
\newcommand{\ZF}{\ifmmode\mathsf{ZF}\else$\mathsf{ZF}$\fi}
\newcommand{\ZFC}{\ifmmode\mathsf{ZFC}\else$\mathsf{ZFC}$\fi}
\newcommand{\AC}{\ifmmode\mathsf{AC}\else$\mathsf{AC}$\fi}
\newcommand{\AD}{\ifmmode\mathsf{AC}\else$\mathsf{AD}$\fi}
\newcommand{\BZ}{\ifmmode\mathsf{BZ}\else$\mathsf{BZ}$\fi}
\newcommand{\GCH}{\ifmmode\mathsf{GCH}\else$\mathsf{GCH}$\fi}
\newcommand{\PA}{\ifmmode\mathsf{PA}\else$\mathsf{PA}$\fi}
\newcommand{\DC}{\ifmmode\mathsf{DC}\else$\mathsf{DC}$\fi}
\newcommand{\MP}{\ifmmode\mathsf{MP}\else$\mathsf{MP}$\fi}
\newcommand{\CT}{\ifmmode\mathsf{CT}\else$\mathsf{CT}$\fi}
\newcommand{\PAx}{\ifmmode\mathsf{PAx}\else$\mathsf{PAx}$\fi}
\newcommand{\VL}{\ifmmode{\mathsf{V}=\mathsf{L}}\else$\mathsf{V}=\mathsf{L}$\fi}

\newcommand{\res}{{\upharpoonright}}

\newlength{\savedparindent}
\AtBeginDocument{\setlength{\savedparindent}{\parindent}}

\DeclareMathOperator{\tp}{tp}
\DeclareMathOperator{\cf}{cf}

\newcommand{\Fraisse}{Fra\"\i ss\'e}

\def\Ind{\setbox0=\hbox{$x$}\kern\wd0\hbox to 0pt{\hss$\mid$\hss}
\lower.9\ht0\hbox to 0pt{\hss$\smile$\hss}\kern\wd0}

\def\Notind{\setbox0=\hbox{$x$}\kern\wd0\hbox to 0pt{\mathchardef
\nn=12854\hss$\nn$\kern1.4\wd0\hss}\hbox to
0pt{\hss$\mid$\hss}\lower.9\ht0 \hbox to 0pt{\hss$\smile$\hss}\kern\wd0}

\def\ind{\mathop{\mathpalette\Ind{}}}

\newcommand{\indu}{\ind^{\!\!\textnormal{u}}}

\newcommand{\indbu}{\ind^{\!\!\textnormal{bu}}}

\DeclareMathOperator{\lex}{lex}

\newcommand{\trianglegeq}{\mathrel{\trianglerighteq}}
\newcommand{\trianglegess}{\mathrel{\triangleright}}
\newcommand{\triangleleq}{\mathrel{\trianglelefteq}}
\newcommand{\triangleless}{\mathrel{\triangleleft}}

\DeclareMathOperator{\bu}{bu}

\newcommand{\Equiv}{\mathsf{Equiv}}

\newcommand{\NOP}{\mathrm{NOP}}
\newcommand{\OP}{\mathrm{OP}}
\newcommand{\NFOP}{\mathrm{NFOP}}
\newcommand{\FOP}{\mathrm{FOP}}
\newcommand{\NIP}{\mathrm{NIP}}
\newcommand{\IP}{\mathrm{IP}}

\newcommand{\IFOP}{\mathrm{IFOP}}
\newcommand{\NSOP}{\mathrm{NSOP}}

\newcommand{\NTP}{\mathrm{NTP}}

\newcommand{\concat}{\mathop{\frown}}

\begin{document}

\title{Indiscernible extraction at small large cardinals from a higher-arity stability notion}

\address{Department of Mathematics \\
  Iowa State University \\
  396 Carver Hall \\
  411 Morrill Road \\
  Ames, IA 50011, USA}
\author{James E. Hanson}
\email{jameseh@iastate.edu}
\date{\today}


\begin{abstract}
  We introduce a higher-arity stability notion defined in terms of \emph{$k$-splitting}, a higher-arity generalization of splitting. We show that theories with bounded $k$-splitting have improved indiscernible extraction at $k$-ineffable cardinals, and we give a non-trivial example of a theory with bounded $k$-splitting but unbounded $(k-1)$-splitting for each odd $k > 1$. We also show that bounded $k$-splitting implies $\NFOP_k$, a higher arity stability notion introduced by Terry and Wolf. We then use our indiscernible extraction result together with a construction of Kaplan and Shelah to give a strong counterexample to the converse: an $\NIP$ theory with unbounded $k$-splitting for every $k$. Finally, as a thematically related but technically independent result, we show that treelessness implies $\NFOP_2$, sharpening a result of Kaplan, Ramsey, and Simon.

\end{abstract}

\maketitle

\section*{Introduction}
\label{sec:intro}

In \cite{Shelah_2014},\footnote{See also \cite{Chernikov_2019} for a comprehensive treatment of $k$-dependence.} Shelah introduced a natural higher-arity generalization of the independence property, the negation of which is called \emph{$k$-dependence} or $\NIP_k$. Since then, some model theorists and combinatorialists have tried to find good higher-arity generalizations of stability, with the most developed general notion currently in print being the (negation of the) \emph{functional order property} or (N)$\FOP_k$, introduced by Terry and Wolf in %
\cite{TerryWolf2021,Terry_Wolf_Irregular} but also developed extensively in the infinitary model-theoretic setting by Abd Aldaim, Conant, and Terry in \cite{FOP}. In the binary case, Takeuchi defined in \cite{Takeuchi2OrderProperty} a similar notion which he called the \emph{$2$-order property} or $\OP_2$ and which is referred to in \cite{FOP,TerryWolf2021} as the \emph{increasing functional order property} or $\IFOP_2$. In \cite{Kaplan_2024}, Kaplan, Ramsey, and Simon also introduced the notion of \emph{treelessness}, a mutual generalization of stability and binarity, and demonstrated its model-theoretic significance by, among other things, showing that any treeless $\NSOP_1$ theory is simple. Chernikov has also introduced a similar `generic binarity' condition called \emph{$\Cc$-lessness} \cite{ChernikovUpcoming2025,ChernikovOberwolfach2023}.\footnote{In \cite{ChernikovOberwolfach2023}, this condition is also referred to as treelessness, but the precise relationship between treelessness (in the sense of \cite{Kaplan_2024}) and $\Cc$-lessness is unknown.} These notions have a tight relationship with the $\NIP_k$ hierarchy (see Figure~\ref{fig:implications}).

Generally speaking, a $k$-ary stability notion needs to be a reasonably well-motivated mutual generalization of stability and $k$-arity.\footnote{Recall that a theory $T$ is \emph{$k$-ary} (or satisfies the property of \emph{$k$-arity}) if every formula is equivalent modulo $T$ to a Boolean combination of formulas with at most $k$ free variables. A theory is \emph{binary} if it is $2$-ary.} In this paper we are going to define a higher-arity stability notion (\cref{defn:bounded-k-splitting}) by applying this philosophy to one of the basic motivation concepts in Shelah's original development of stability theory, namely \emph{splitting}. As we will see in \cref{sec:relationship-to-other}, however, the notion we arrive at doesn't seem to fit cleanly into the hierarchy in Figure~\ref{fig:implications}. 

Recall that a type $\tp(a/B) \in S_x(B)$ \emph{splits} over $C \subseteq B$ if there is a formula $\varphi(x,y)$ and $b,b' \in A$ with $b \equiv_C b'$ such that $\varphi(a,b)$ and $\neg \varphi(a,b')$ hold \cite[Def.~I.2.6]{shelah1990classification}. %

\begin{fact}[{\cite[Lem.~I.2.7]{shelah1990classification}}]\label{fact:stable-split}
  $T$ is stable if and only if there is a cardinal $\lambda$ such that for any $a$ and $B$, there is a $C \subseteq B$ with $|C| < \lambda$ such that $\tp(a/B)$ does not split over $C$.
\end{fact}

Splitting is quite important in Shelah's original conception of stability theory. For example, Shelah defines forking in terms of `strong splitting' at the beginning of Chapter III of \cite{shelah1990classification} and then introduces dividing a few pages later.

Shelah doesn't use the following terminology in \cite{shelah1990classification}, but the concept is implicit in the statement of \cite[Lem.~I.2.5]{shelah1990classification}.

\begin{defn}
  A sequence $(a_i)_{i < \alpha}$ of elements is \emph{end-homogeneous over $B$} if for any $i < j < \alpha$ and $a_i \equiv_{Ba_{<i}} a_j$.
\end{defn}


%

 The significance of this condition is that the task of finding an indiscernible sequence can be factored into non-splitting and end-homogeneity.

\begin{fact}[{Shelah \cite[Lem.~I.2.5]{shelah1990classification}}]\label{fact:indisc-tail-split}
  If $(a_i)_{i < \alpha}$ is end-homogeneous over $B$ and for each $i < \alpha$, $\tp(a_i/Ba_{<i})$ does not split over $B$, then $(a_i)_{i <\alpha}$ is indiscernible over $B$.
\end{fact}

This fact is essentially why stability allows for improved indiscernible extraction. Bounded splitting ensures that any sufficiently long end-homogeneous sequence will be indiscernible on a large subsequence, and, unlike ordinary indiscernibility, end-homogeneity is easy to find lying around:

\begin{fact}\label{lem:tail-indisc-extraction}
  For any set of parameters $B$, $\lambda \geq |T|+|B|$, and sequence $(a_i)_{i < \beth(\lambda)^+}$, there is an $X \subseteq \beth(\lambda)^+$ with $|X| = \lambda^+$ such that $(a_i)_{i \in X}$ is end-homogeneous over $B$.
\end{fact}
\begin{proof}
  This is essentially a step of a common proof of the Erd\H{o}s-Rado theorem. Define a sequence $(\beta_j)_{j < \lambda^+}$ of ordinals less than $\beth(\lambda)^+$ inductively by letting $\beta_j$ be the smallest ordinal less than $\beth(\lambda)^+$ and greater than $\sup_{\ell < j}\beta_\ell$ satisfying that for any $Y \subseteq \sup_{\ell < j} \beta_\ell$ with $|Y| \leq \lambda$, every type\footnote{This step in particular can be improved in the context of a $\lambda$-stable theory.} over $B\cup \{a_i : i \in Y\}$ that is realized by an element of the sequence $(a_i)_{i < \beth(\lambda)^+}$ is realized by some $a_i$ with $i < \beta_j$. (Such a $\beta_j < \beth(\lambda)^+$ always exists.) Since $\lambda^+ \leq \beth(\lambda)$, the sequence $(\beta_{j})_{j < \lambda^+}$ is not cofinal in $\beth(\lambda)^+$, so let $\gamma = \sup_{j < \lambda^+}\beta_j$. For each $j < \kappa^+$, let $\delta_j$ be the least ordinal greater than $\beta_j$ such that $a_{\delta_j} \equiv_{Ba_{<\beta_j}} a_\gamma$. (Such a $\delta_j$ must always exist by construction.) Let $X = \{\delta_j : j < \kappa^+\}$. Clearly $|X| = \kappa^+$ and it is easy to check that the sequence $(a_i)_{i \in X}$ is end-homogeneous over $B$.
\end{proof}

\begin{figure}
  \centering
  \adjustbox{scale=0.9}{
  \begin{tikzcd}[cramped]
      & {\text{NIP}} \\
      {\text{Stability}} & {\Cc\text{-lessness}} & {\text{NOP}_2} & {\text{NFOP}_2} & {\text{NIP}_2} & {\text{NFOP}_3} & {\text{NIP}_3} & \cdots \\
      & {\text{Treelessness}} && {\cdots} & {\text{NFOP}_k} & {\text{NIP}_k} & {\text{NFOP}_{k+1}} & {\cdots}
      \arrow["\text{\cite{Takeuchi2OrderProperty}}", from=1-2, to=2-3]
      \arrow[from=2-1, to=1-2]
      \arrow["\text{\cite[Cor~5.4]{Kaplan_2024}}"', from=2-1, to=3-2]
      \arrow["\text{Easy}", from=2-3, to=2-4]
      \arrow[from=2-4, to=2-5]
      \arrow[from=2-5, to=2-6]
      \arrow[from=2-6, to=2-7]
      \arrow[from=2-7, to=2-8]
      \arrow["\text{{Prop.~\ref{prop:treeless-NOP}}}"', from=3-2, to=2-3]
      \arrow[from=3-4, to=3-5]
      \arrow[from=3-5, to=3-6]
      \arrow[from=3-6, to=3-7]
      \arrow[from=3-7, to=3-8]
      \arrow[from=2-1, to=2-2]
      \arrow[from=2-1, to=2-2]
      \arrow[from=2-2, to=2-3]
    \end{tikzcd}
  }
  \caption{Implications between higher-arity stability notions. The implications involving $\Cc$-lessness are in \cite{ChernikovUpcoming2025}. All implications after $\NFOP_2$ are \cite[Prop.~2.8]{FOP}.}
  \label{fig:implications}
\end{figure}
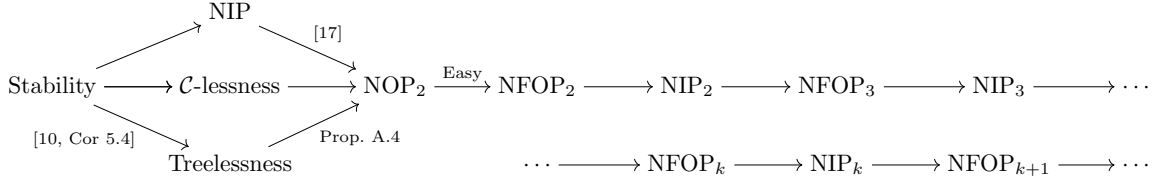

We can use this to easily prove a sloppy version of Shelah's theorem \cite[Thm.~I.2.8]{shelah1990classification} on indiscernible extraction in stable theories:

\begin{prop}\label{prop:indisc-extract-sloppy}
  Let $T$ be a stable theory with $\kappa$ a cardinal witnessing this in the sense of \cref{fact:stable-split}. Let $B$ be a set of parameters in a model of $T$. Let $\lambda$ be a fixed cardinal satisfying that $\lambda \geq |T|+|B|+\kappa$. For any sequence $(a_i)_{i < \beth(\lambda)^+}$, there is a set $X \subseteq \beth(\lambda)^+$ with $|X| = \lambda^+$ such that $(a_i)_{i \in X}$ is indiscernible over $B$.
\end{prop}
\begin{proof}
  By \cref{lem:tail-indisc-extraction}, we can find a set $Y \subseteq \beth(\lambda)^+$ with $|Y|= \lambda^+$ such that $(a_i)_{i \in Y}$ is end-homogeneous over $B$. Re-index this as a sequence $(a_i)_{i < \kappa^+}$.

  For each $i < \kappa^+$, let $f(i)$ be the least $j \leq i$ such that $\tp(a_i/Ba_{<i})$ does not split over $B$. The set $S = \{i < \lambda^+ : \cf(i) \geq \kappa\}$ is stationary in $\lambda$ (since $\cf(\lambda^+) = \lambda^+ > \kappa$). By the choice of $\kappa$, we have that $f(i)$ is a regressive function on $S$, so, by Fodor's lemma, there is a $\gamma < \lambda$ and a stationary set $X \subseteq S$ such that $f(i) = \gamma$ for any $i \in X$. We now have that for all $i \in X$, $\tp(a_i/B\cup a_{<\min X}\cup\{a_j : j < i,~j \in X\})$ does not split over $B\cup a_{<\min X}$. Therefore, by \cref{fact:indisc-tail-split}, we have that $(a_i)_{i \in X}$ is indiscernible over $Ba_{< \min X}$ and so indiscernible over $B$.
\end{proof}

Of course the actual theorem is a more precise result in that we are able to extract an indiscernible sequence of length $\lambda^+$ from any sequence of length $\lambda^+$ whenever $T$ is $\lambda$-stable, but it's important to point out that even \cref{prop:indisc-extract-sloppy} is non-trivial. Unless there is an $\omega$-Erd\H{o}s cardinal, there isn't even an analog of \cref{prop:indisc-extract-sloppy} for arbitrary theories with the weaker conclusion that $|X| = \aleph_0$.

There is a second kind of theory that has indiscernible extraction without the assumption of large cardinals. In a $k$-ary theory, a sequence is indiscernible if and only if it is $k$-indiscernible (i.e., satisfies that any two increasing subsequences of length $k$ have the same type). The Erd\H{o}s-Rado theorem easily gives the following.

\begin{fact}[Erd\H{o}s-Rado]\label{fact:Erdos-Rado-extraction}
  For any $k$-ary theory $T$ and sequence of elements $(a_i)_{i < \beth_{k}(|T|)^+}$, there is a set $X \subseteq \beth_{k}(|T|)^+$ with $|X| = \left(2^{|T|}\right)^+$ such that $(a_i)_{i \in X}$ is indiscernible.
\end{fact}

There is a relationship between Facts~\ref{fact:stable-split} and \ref{fact:Erdos-Rado-extraction}, which is that Fodor's lemma and the Erd\H{o}s-Rado theorem are in some sense different generalizations of the infinitary pigeonhole principle. Erd\H{o}s-Rado, like all Ramsey-like theorems, is a higher-arity generalization, and Fodor's lemma is a generalization that allows for an `increasing' set of pigeonholes. There are further mutual generalizations of these (i.e., partition properties that involve both higher arity and an increasing set of `hyper-pigeonholes'), but like other strong partition properties, they end up characterizing certain kinds of large cardinals.

\section{$k$-ineffability and $k$-end-homogeneity}
\label{sec:ineff-and-tail}

\begin{defn}[{Baumgartner \cite{Baumgartner1975}}]
  A function $f : [\lambda]^k \to 2^\mu$ is \emph{regressive} if $f(A) \subseteq \min A$ for all $A \in [\lambda]^k$. Given such a function $f$, a set $X \subseteq \lambda$ is \emph{$f$-homogeneous} if for any $A,B \in [\lambda]^k$, $f(A) \cap \min(A \cup B) = f(B) \cap \min(A\cup B)$.

  A limit cardinal $\mu$ is \emph{$k$-ineffable} if for any regressive function $f : [\mu]^k \to 2^\mu$, there is an $f$-homogeneous stationary set $A \subseteq \mu$. An \emph{ineffable cardinal} is a $2$-ineffable cardinal.
\end{defn}

The theory of $k$-ineffable cardinals was originally developed by Baumgartner \cite{Baumgartner1975}, but Friedman's paper \cite{Friedman_2001} is easier to find, and as such we will be citing Friedman in this paper. Any $k$-ineffable cardinal is strongly inaccessible \cite[Lem.~10]{Friedman_2001}. It's also fairly straightforward to show that if $\mu$ is $k$-ineffable, then for any $\lambda < \mu$, $\mu$ satisfies the partition relation $\mu \to (\mu)^k_{\lambda}$.

Given a sequence $(a_i)_{i < \alpha}$ and a tuple of indices $\ibar = i_0,\dots,i_{n-1}$, we will sometimes write $a_{\ibar}$ for the tuple $a_{i_0},\dots,a_{i_{n-1}}$.

\begin{defn}
  A sequence $(a_i)_{i < \mu}$ of elements is \emph{$k$-end-homogeneous over $B$} if for any $i < \mu$ and increasing $k$-tuples $\nbar,\mbar > i$, $a_{\nbar} \equiv_{Ba_{<i}} a_{\mbar}$.
\end{defn}

Some questions about $k$-end-homogeneity in a model-theoretic context are asked in \cite{Shelah2012}.

It is relatively straightforward to modify a proof of the Erd\H{o}s-Rado theorem (such as the one of \cite[Thm.~C.3.2]{tent_ziegler_2012}) to generalize \cref{lem:tail-indisc-extraction} to $k$-end-homogeneity. In particular, one can show that for any theory $T$, set of parameters $B$, and sequence $(a_i)_{i < \beth_k(|T|+|B|)^+}$, there is an $X \subseteq \beth_k(|T|+|B|)^+$ of order type $\beth(|T|+|B|)^+$ such that $(a_i)_{i \in X}$ is $k$-end-homogeneous over $B$. Our ultimate goal, however, is to extract an indiscernible sequence that is the same size as the original sequence, so we need a stronger result.

\begin{prop}\label{prop:ineff-tail}
  For any $\Lc$-theory $T$, set of parameters $B$, any $k$-ineffable cardinal $\mu > |T| + |B|$, and any sequence $(a_i)_{i < \mu}$ of $\gamma$-tuples (with $\gamma < \mu$) in a model of $T$, there is a stationary set $X \subseteq \mu$ such that $(a_i)_{i \in X}$ is $k$-end-homogeneous over $B$.
\end{prop}
\begin{proof}
  By adding $B$ as a set of constants to $T$, we may assume without loss of generality that $B = \varnothing$. Fix a $k$-tuple of $\gamma$-tuples of variables $\xbar = x_0\dots x_{k-1}$. We will define a sequence $(\chi_j)_{j < \mu}$ inductively such that for each $i < \mu$, $\chi_j$ is either an element of the sequence $(a_i)_{i < \mu}$ or a formula. Moreover, we will define an increasing sequence $(j_i)_{i < \mu}$ of indices such that $\chi_{j_i} = a_i$ for each $i < \mu$.

  To do this, suppose we are at some stage $\alpha < \mu$ and suppose that we have defined $j_i$ for all $i < \alpha$ and $\chi_j$ for $j$ in some initial segment of $\mu$. Let $j_\alpha$ be the least element of $\mu$ such that $\chi_{j_\alpha}$ is not defined. Let $\chi_{j_\alpha} = a_\alpha$. Find an ordinal $\beta < \mu$ with $\beta > j_\alpha$ such that the cardinality of $\beta \setminus j_\alpha$ is the same as the cardinality of the set of formulas in the variables $\bar{x}$ with parameters among $a_{< \alpha}$. (Note that this is always possible since $\mu$ is a strongly inaccessible cardinal. Also note that $\beta$ as defined here will be $j_{\alpha + 1}$ in the next stage of the construction.) Define $\chi_{j}$ for $j \in \beta \setminus j_\alpha$ such that $(\chi_j)_{j_\alpha < j < \beta}$ is an enumeration of these formulas. Then proceed to the next stage of the construction.

  Note that in the above construction we have ensured that the set $\{j_i : i < \mu\}$ is a club (because for limit $i < \mu$, $j_i = \sup\{j_{i'} : i' < i\}$).

  Now define a regressive function $f : [\mu]^k\to 2^\mu$ in the following way:
  \begin{itemize}
  \item If any of the elements of $A \in [\mu]^k$ are formulas, then $f(A) = \varnothing$.
  \item If $A \in [\mu]^k$ is $\{\chi_{j_{i_0}},\dots,\chi_{j_{i_{k-1}}}\} = \{a_{i_0},\dots,a_{i_{k-1}}\}$ (in increasing order), then $f(A)$ is the set of $j < j_{i_0}$ such that $\chi_j$ is a formula $\varphi(x_0,\dots,x_{k-1})$ such that $\varphi(a_{i_0},\dots,a_{i_{k-1}})$ holds.
  \end{itemize}

  Now since $\mu$ is $k$-ineffable, we can find an $f$-homogeneous stationary set $Y \subseteq \mu$. Since $\{j_i : i < \mu\}$ is a club, we may assume that $Y \subseteq \{j_i : i < \mu\}$. Let $X = \{i < \mu : j_i \in Y\}$, and note that $X$ is also a stationary set. It is now immediate by construction that for any two increasing $k$-tuples of indices $i_0,\dots,i_{k-1}$ and $i'_0,\dots,i'_{k-1}$, the tuples $a_{i_0},\dots,a_{i_{k-1}}$ and $a_{i'_0},\dots,a_{i'_{k-1}}$ satisfy the same formulas over $a_{<\min\{i_0,i'_0\}}$, so the sequence $(a_i)_{i \in X}$ is $k$-end-homogeneous, as required.
\end{proof}

One thing to note is that there is a mismatch between indiscernible extraction in the context of stability and what we are doing here. In stability, one leverages $\lambda$-stability to show that one can extract a $1$-end-homogeneous sequence $(a_i)_{i \in X}$ from any sequence $(a_i)_{i < \lambda^+}$ (improving \cref{lem:tail-indisc-extraction}). Here we are leaning on $k$-ineffability since we seemingly need it as a replacement for Fodor's lemma in the proof of \cref{thm:bounded-k-splitting-indisc-extraction} anyway.

\section{Yet another higher-arity stability notion: Bounded $k$-splitting}
\label{sec:bounded-k-splitting}

We will adopt the common convention that given a list $a_0,\dots,a_{k-1}$, the expression $a_0,\dots,\hat{a}_i,\dots,a_{k-1}$ denotes the same list with the element $a_i$ removed. We will be using this notation frequently, such as in the list of variables of a type (i.e., $p(x_0,\dots,\hat{x}_i,\dots,x_{k-1})$ is a type in the variables $x_0,\dots,x_{i-1},x_{i+1},\dots,x_{k-1}$). We will also often write tuples of variables $\xbar$ as just $x$ without an over-bar, unless it is important to emphasize the fact that $\xbar$ is a tuple of variables.

\begin{defn}\label{defn:k-splitting}
  A \emph{$k$-partitioned type} is a type $p(x_0;\dots;x_{k-1})$ with a designated partition of its variables into $k$ sets.

  A $k$-partitioned type $\tp(a_0;\dots;a_{k-1}/B)$ \emph{$k$-splits} over $C \subseteq B$ if there is a formula $\varphi(x_0,\dots,x_{k-1},y)$ and $b,b' \in B$ such that
  \begin{itemize}
  \item for each $i < k$, $b \equiv_{Ca_0\dots\hat{a}_i\dots a_{k-1}} b'$ and
  \item $\varphi(a_0,\dots,a_{k-1},b)$ and $\neg \varphi(a_0,\dots,a_{k-1},b')$ hold.
  \end{itemize}
\end{defn}

As one would hope, we get a generalization of \cref{fact:indisc-tail-split} with $k$-splitting and $k$-end-homogeneity.

\begin{prop}\label{prop:k-indisc-tail-split}
  For any $k < \omega$ and $\alpha \geq k$, if $(a_i)_{i < \alpha}$ is $k$-end-homogeneous over $B$ and for each $i_0 < \dots < i_{k-1} < \alpha$, $\tp(a_{i_0};\dots;a_{i_{k-1}}/Ba_{<i_0})$ does not $k$-split over $B$, then $(a_i)_{i < \alpha}$ is indiscernible over $B$.
\end{prop}
\begin{proof}
  Assume that $(a_i)_{i < \alpha}$ is $k$-end-homogeneous over $B$ and has the non-$k$-splitting property in the statement of the proposition. Clearly the sequence is already $k$-indiscernible over $B$ (i.e., satisfies that any two increasing subsequences of length $k$ have the same type over $B$). We will prove by induction that the sequence is $n$-indiscernible over $B$ for every $n < \omega$.

  Assume that we have shown that the sequence is $(n+k)$-indiscernible over $B$ for some $n<\omega$. Fix $i_0< \dots < i_{n+k} < \alpha$. We need to show that $a_{i_0}\dots a_{i_{n+k}} \equiv_B a_{0}\dots a_{n+k}$.

  By $k$-end-homogeneity, we have that
  \[
    a_0\dots a_{n-1}a_{n}\dots a_{n+k} \equiv_B a_0\dots a_{n}a_{i_{n+1}}\dots a_{i_{n+k}}.
  \]
  Now considered the $k$-partitioned type $p = \tp(a_{i_{n+1}};\dots;a_{i_{n+k}}/Ba_{<i_{n+1}})$. By $n$-indiscernibility, we have that for each $j \in \{1,\dots,k\}$,
  \[
    a_0\dots a_{n} \equiv_{B a_{n+1}\dots \hat{a}_{n+j}\dots a_{n+k}} a_{i_0} \dots a_{i_{n}}.
  \]
  Therefore, since $p$ does not $k$-split over $B$, we have that for any $B$-formula $\varphi(x_0,\dots,x_{k},y_{0},\dots,y_{n})$,
  \[
    \varphi(a_{i_{n+1}},\dots,a_{i_{n+k}},a_0,\dots,a_{n}) \leftrightarrow \varphi(a_{i_{n+1}},\dots,a_{i_{n+k}},a_{i_0},\dots,a_{i_n}),
  \]
  whereby $a_0\dots a_na_{i_{n+1}}\dots a_{i_{n+k}} \equiv_B a_{i_0}\dots a_{i_{n}}a_{i_{n+1}}\dots a_{i_{n+k}}$. Therefore $a_0\dots a_{n+k}$ and $a_{i_0}\dots a_{i_{n+k}}$ have the same type over $B$. Since we can do this for any such increasing $(n+1+k)$-tuple of indices, we have that the sequence is $(n+1+k)$-indiscernible over $B$. Therefore, by induction, the sequence is indiscernible over $B$.
\end{proof}

Now it would seem to make sense to introduce the following definition by direct analogy with \cref{fact:stable-split}.

\begin{defn}\label{defn:bounded-k-splitting}
  Given an ordinal $\gamma$, a theory $T$ has \emph{$\lambda$-bounded $k$-splitting for $\gamma$-tuples} if for any $k$-partitioned type $\tp(\abar_0;\dots;\abar_{k-1}/B)$ (with each $\abar_i$ a $\gamma$-tuple), there is a $C \subseteq B$ with $|C| < \lambda$ such that $\tp(\abar_0;\dots;\abar_{k-1}/B)$ does not $k$-split over $C$. $T$ has \emph{$\lambda$-bounded $k$-splitting} if for any finite $n$, $T$ has $\lambda$-bounded $k$-splitting for $n$-tuples.  $T$ has \emph{bounded $k$-splitting} if it has $\lambda$-bounded $k$-splitting for some $\lambda$. $T$ has \emph{unbounded $k$-splitting} if it does not have bounded $k$-splitting.

\end{defn}

It is immediate that $k$-arity implies bounded $k$-splitting.

\begin{prop}
  If $T$ is a $k$-ary theory, then it has $1$-bounded $k$-splitting.
\end{prop}
\begin{proof}
  The $k$-arity condition implies that for any partitioned $k$-type $\tp(a_0;\dots;a_{k-1}/B)$ and any $\bbar \in B$, the type $\tp(\bbar,\abar)$ is isolated by
  \[
    \tp(a_0,\dots,a_{k-1})\cup\tp(\bbar,\hat{a}_0,\dots,a_{k-1}) \cup \cdots \cup \tp(\bbar,a_0,\dots,\hat{a}_{k-1}),
  \]
  whereby $\tp(a_0;\dots;a_{k-1}/B)$ does not $k$-split over $\varnothing$.
\end{proof}

It is also relatively straightforward to see that bounded $k$-splitting implies bounded $(k+1)$-splitting.

\begin{lem}\label{lem:k-down}
  If $\tp(a_0;\dots;a_{k-1}/B)$ $k$-splits over $C$, then $\tp(a_0;\dots;a_{k-2}/Ba_{k-1})$ $(k-1)$-splits over $Ca_{k-1}$.
\end{lem}
\begin{proof}
  This is immediate from the definition.
\end{proof}

\begin{prop}\label{prop:next-k}
  If $T$ has $\lambda$-bounded $k$-splitting, then it has $\lambda$-bounded $(k+1)$-splitting.
\end{prop}
\begin{proof}
  \begin{sloppypar}
    Fix a $(k+1)$-partitioned type $\tp(a_0;\dots;a_{k}/B)$. There is a set $C$ with $|C|< \lambda$ such that $\tp(a_0;\dots;a_{k-1}/Ba_{k})$ does not $k$-split over $C$ and therefore does not $k$-split over $Ca_{k}$. Therefore by \cref{lem:k-down}, $\tp(a_0;\dots;a_{k}/B)$ does not $(k+1)$-split over $C$. \qedhere
  \end{sloppypar}
\end{proof}

So in this way we can see that bounded $k$-splitting is a reasonable mutual generalization of stability and $k$-arity, since stability is equivalent to bounded $1$-splitting by \cref{fact:stable-split}.

In the case of stability, of course, there is a strong bound on how big $\lambda$ needs to be (i.e., if $T$ has bounded $1$-splitting, then it has $|T|^{+}$-bounded $1$-splitting). We can prove the analogous bound for $k$-splitting fairly easily, which we will use later in \cref{sec:relationship-to-other} to establish that there are $\NIP$ theories with unbounded $k$-splitting for every $k$. We will also take the opportunity to prove that bounded $k$-splitting for types in finitely many variables entails bounded $k$-splitting for types in infinitely many variables.

\begin{prop}\label{prop:bound-on-bounded-k-splitting}
  If $T$ has bounded $k$-splitting, then for any ordinal $\gamma$, $T$ has $(|T|+|\gamma|)^+$-bounded $k$-splitting for $\gamma$-tuples. In particular, if $T$ has bounded $k$-splitting, then it has $|T|^{+}$-bounded $k$-splitting.
\end{prop}
\begin{proof}
  Fix an ordinal $\gamma$ and let $\delta = |T|+|\gamma|$. Suppose that $T$ does not have $\delta^+$-bounded $k$-splitting for $\gamma$-tuples. Then there is a set $B$ and a $k$-partitioned type $\tp(a_0;\dots;a_{k-1}/B) \in S_{\xbar}(B)$ (with each $a_i$ a $\gamma$-tuple) such that $p$ $k$-splits over every $C \subseteq B$ with $|C| \leq \delta$. Define sequences $(b_j)_{j < \delta^+}$, $(b'_j)_{j < \delta^+}$, and $(\varphi_j(\xbar,\ybar))_{j<\delta^+}$ inductively so that for each $j < \delta^+$,
  \begin{itemize}
  \item for each $i < k$, $b_j \equiv_{a_0\dots \hat{a}_i \dots a_{k-1}b_{<j}c_{<j}} b'_j$ and
  \item $\varphi_j(a_0,\dots,a_{k-1},b_j)$ and $\neg \varphi_j(a_0,\dots,a_{k-1},b'_j)$ hold.
  \end{itemize}
  Since $\delta^+$ is a regular cardinal, there must be some fixed formula $\varphi$ such that $|\{j < \delta^+ : \varphi_j = \varphi\}| = \delta^+$. Therefore, we may assume that $\varphi_j$ is always equal to some fixed formula $\varphi$. Moreover, we may restrict the tuples $a_i$ to some finite subtuple. By compactness we can stretch this configuration to arbitrary length (since the above bulleted items and the fact that $\{\varphi(\xbar,b_j) : j < \delta^+\}\cup \{\neg \varphi(\xbar,\neg b'_j) : j < \delta^+\}$ is finitely consistent are part of the type of $b_{<\delta^+}b'_{<\delta^+}$), so for any $\lambda$, we can find a set of parameters $B$ such that that some type $p(\xbar_0;\dots;\xbar_{k-1})$ (now with each $\xbar_i$ a tuple in finitely many variables) $k$-splits over any $C \subseteq B$ with $|C| < \lambda$.

  The final statement follows by considering the case of finite $\gamma$.
\end{proof}

Unbounded $k$-splitting for $k > 1$ doesn't seem to lend itself to being `unfolded' into a combinatorial configuration in the same way that unbounded $1$-splitting can be shown to be equivalent to the order property. This is also a barrier to trying to characterize bounded $k$-splitting in terms of some kind of type counting condition. This raises the following question.

\begin{quest}
  Can bounded $k$-splitting be characterized by a more traditional combinatorial configuration (such as the \emph{$n$-patterns} of \cite{Bailetti2024})? Can it be characterized in terms of type counting?
\end{quest}

\section{Indiscernible extraction at $k$-ineffable cardinals from bounded $k$-splitting}
\label{sec:indisc-extraction-from-bounded-k-splitting}

Now we will prove our main theorem regarding extraction of indiscernible sequences at $k$-ineffable cardinals.

Recall that G\"odel defined a pairing function $\langle \cdot,\cdot \rangle : \mathrm{Ord} \times \mathrm{Ord} \to \mathrm{Ord}$ with the property that any cardinal $\lambda$ is closed under $\langle \cdot,\cdot \rangle$. Note moreover that for any cardinal $\lambda$, the set of $\alpha < \lambda$ that are closed under $g$ is a club. We will refer to an ordinal closed under $\langle \cdot,\cdot \rangle$ as \emph{$\langle \cdot,\cdot \rangle$-closed}.

\begin{thm}\label{thm:bounded-k-splitting-indisc-extraction}
  If $T$ has bounded $k$-splitting and $\mu > |T| + |B| + |\gamma|$ is $k$-ineffable, then for any sequence $(a_i)_{i < \mu}$ of $\gamma$-tuples of parameters, there is an $X \subseteq \mu$ with $|X| = \mu$ such that $(a_i)_{i \in X}$ is indiscernible over $B$.
\end{thm}
\begin{proof}
  We may assume without loss of generality that $B = \varnothing$. Let $\delta = |T|+|\gamma|$. Note that by \cref{prop:bound-on-bounded-k-splitting}, we have that $T$ has $\delta^+$-bounded $k$-splitting for $\gamma$-tuples. Since $\mu$ is strongly inaccessible, $\mu > \delta^+$ as well. By \cref{prop:ineff-tail}, we may assume without loss of generality that $(a_i)_{i < \mu}$ is $k$-end-homogeneous.

  Define a function $f : [\mu]^k \to 2^\mu$ as follows: Fix $I = \{i_0,\dots,i_{k-1}\}$ with $i_0 < \dots i_{k-1} < \mu$.
  \begin{itemize}
  \item If $i_0 \leq |T|^+$ or if $i_0$ is not $\langle \cdot,\cdot \rangle$-closed, then let $f(I) = \varnothing$.
  \item If $i_0 > |T|^+$ and $i_0$ is $\langle \cdot,\cdot \rangle$-closed, then let $C_I \subseteq i_0$ be some non-empty set of indices with $|C_I| \leq |T|^+$ satisfying that $\tp(a_{i_0};\dots;a_{i_{k-1}}/a_{<i_0})$ does not $k$-split over $\{a_j : j \in C_I\}$. Let $h_I : |T|^+ \to C_I$ be a surjection, and let $f(I) = \{\langle j,h_I(j) \rangle : j < |T|^+\}$. (Note that since $i_0$ is $\langle  \cdot, \cdot \rangle$-closed, we have that $f(I) \subseteq i_0 = \min I$.)
  \end{itemize}
  Since $\mu$ is $k$-ineffable, we can find a stationary set $X \subseteq \mu$ that is $f$-homogeneous. Since the set of $\langle \cdot,\cdot \rangle$-closed ordinals is a club in $\mu$, we may assume that every element of $X$ is $\langle \cdot,\cdot \rangle$-closed. Since $X$ is stationary, it must be unbounded in $\mu$, we may also assume that every element of $X$ is greater than $|T|^+$.

\vspace{0.5em}

\begin{claim}{}
  For any $I,J \in [X]^k$, $h_I = h_J$, and so in particular $C_I = C_J$.
\end{claim}
\begin{claimproof}
  Assume without loss of generality that $\min I \leq \min J$. By the definition of $f$-homogeneity, we have that $f(I) = f(J)\cap \min I$. Since $\langle \cdot,\cdot \rangle$ is a pairing function, this implies that $h_I \subseteq h_J$. Since $h_I$ and $h_J$ are functions with the same domain, we have that $h_I = h_J$.
\end{claimproof}

\vspace{0.5em}

Unpacking what we have done, it follows from the claim that there is a single set $C$ with $|C| = |T|^+$ of indices such that for any $\{i_0,\dots,i_{k-1}\} \in [X]^k$, $\tp(a_{i_0};\dots;a_{i_{k-1}}/a_{<i_0})$ does not $k$-split over $a_{\in C} \coloneq \{a_j : j \in C\}$. Therefore we also have that for any $\{i_0,\dots,i_{k-1}\} \in [X]^k$, $\tp(a_{i_0};\dots;a_{i_{k-1}}/a_{\in C} \cup \{a_j : j < i_0,~j \in X\})$ does not $k$-split over $a_{\in C}$. It is also immediate that the sequence $(a_i)_{i \in X}$ is $k$-end-homogeneous over $a_{\in C}$, so by \cref{prop:k-indisc-tail-split}, we have that $(a_i)_{i \in X}$ is indiscernible over $a_{\in C}$. Hence, a fortiori, $(a_i)_{i \in X}$ is indiscernible.
\end{proof}

In \cite[Lem.~13]{Friedman_2001}, it is shown that if there is an $\omega$-Erd\H{o}s cardinal $\kappa$, then there is a cardinal $\mu < \kappa$ that is $k$-ineffable for every $k$. This implies that the improved indiscernible extraction in \cref{thm:bounded-k-splitting-indisc-extraction} is non-trivial.

Recall that the partition notation $\kappa \to (\delta)_{T,n}$ (in the sense of Grossberg and Shelah \cite[Def.~3.1(2), p.~208]{Shelah1986-mn}) means that for any sequence of $n$-tuples of length $\kappa$ in a model of $T$, there is an indiscernible sequence of order type $\delta$. %

\begin{quest}\label{quest:improvement}
  Can the partition property established in \cref{thm:bounded-k-splitting-indisc-extraction} be improved? If $T$ has bounded $k$-splitting, for which $\kappa$, $\delta$, and $n$ does the partition relation $\kappa \to (\delta)_{T,n}$ hold? Does improved indiscernible extraction for theories with bounded $k$-splitting require large cardinals?
\end{quest}

It is know that in the development of simple theories, the common uses of the Erd\H{o}s-Rado theorem are unnecessary and can be replaced with Ramsey's theorem \cite{Vasey_2017}, so one might hope that a similarly careful analysis would give an analogous improvement here.

\section{An application}
\label{sec:application}

At the moment we are only able to give one model-theoretic application of \cref{thm:bounded-k-splitting-indisc-extraction} whose conclusion does not involve the concept of bounded $k$-splitting. This application is admittedly a bit underwhelming in that it is only weakening a large cardinal assumption in a result for arbitrary theories in \cite[Thm.~4.22]{BoundedUltra}. In particular, if it turns out that \cite[Thm.~4.22]{BoundedUltra} is provable without any large cardinal assumptions, then the improvements given here are trivial.

We will use the notation and definitions from \cite{BoundedUltra}, except we will systematically replace the prefix $\indbu$- with $\bu$- (e.g., an $\indbu$-Morley sequence is now a $\bu$-Morley sequence) for the sake of reducing visual clutter and improving verbal readability. %

\begin{prop}\label{prop:weak-application-I}
  Fix a theory $T$, a set of parameters $B$, and a (possibly infinite) tuple $a$. If there is a cardinal $\mu$ such that $\mu \to (\omega+\omega)_{T_B,|a|}$, then there is a total $\bu$-Morley sequence $(a_i)_{i < \omega}$ over $B$ with $a_0 = a$.

  In particular, this occurs under any of the following assumptions:
  \begin{enumerate}
  \item\label{case-1} There is a cardinal $\mu$ satisfying $\mu \to (\omega+\omega)^{<\omega}_{\beth(|T|+|Ba|)}$.
  \item\label{case-2} $T$ has bounded $k$-splitting and there is a $k$-ineffable cardinal $\mu > |T|+|Ba|$.
  \item\label{case-4} $T$ is $k$-ary.\footnote{The fact that the $k$-ary case does not require a large cardinal really ought to have been pointed out in \cite{BoundedUltra} but didn't occur to the author at the time.}
  \end{enumerate}

\end{prop}
\begin{proof}
  By \cite[Prop.~4.14]{BoundedUltra}, we can find a $\bu$-spread-out and $s$-indiscernible over $B$ tree $(a_f)_{f \in \Tc_\mu}$ such that for each $f \in \Tc_\mu$, $a_f \equiv_B a$. By assumption, we can find an $X \subseteq \mu$ with order type $\omega+\omega$ such that the sequence $(a_{\zeta^\mu_{\beta}})_{\beta \in X}$ is $B$-indiscernible. Let $Y \subseteq X$ be first $\omega$ elements of $X$. Let $Z_0 \subseteq X \setminus Y$ be the first $\omega$ elements of $X \setminus Y$. Let $\gamma = \min Z_0$ and $Z = Z_0 \setminus \{\gamma\}$. Note that by construction $\sup Y \leq \gamma$.

  It is now immediate that $(a_{\trianglegeq\zeta^\mu_{\gamma}\concat i})_{ i < \omega}$ is a $\bu$-Morley sequence over $B$ which is moreover $B \cup \{a_{\zeta^\mu_\alpha} : \alpha \in Z\}$-indiscernible. Note that $\{a_{\zeta^\mu_\alpha} : \alpha \in Y\}$ is contained in $(a_{\trianglegeq\zeta^\mu_{\gamma}\concat i})_{ i < \omega}$. Therefore by monotonicity of $\indbu$ and \cite[Lem.~4.11]{BoundedUltra}, we have that  $\{a_{\zeta^\mu_\alpha} : \alpha \in Z\} \indbu_B \{a_{\zeta^\mu_\alpha} : \alpha \in Y\}$. By \cite[Thm.~4.9]{BoundedUltra}, we have that $\{a_{\zeta^\mu_\alpha} : \alpha \in Z\}$ is a total $\bu$-Morley sequence. By applying an automorphism fixing $B$ pointwise, we get the required sequence $(a_i)_{i < \omega}$.

  For the specific numbered assumptions, (\ref{case-1}) is immediate, (\ref{case-2}) follows from \cref{thm:bounded-k-splitting-indisc-extraction}, 
  and for (\ref{case-4}), choose some sufficiently large $\mu$ and use the Erd\H{o}s-Rado theorem. %
\end{proof}

\begin{cor}\label{cor:weak-application-II}
  Fix a theory $T$, a set of parameters $B$, and a (possibly infinite) tuple $a$. If any of the assumptions in \cref{prop:weak-application-I} hold, 
  then there is an infinite $B$-indiscernible sequence $I$ with $a \in I$ such that for any $c$, $a$ and $c$ have the same Lascar strong type over $B$ if and only if there are infinite sequences $J_0,K_0,J_1,\dots,K_{n-1},J_n$ such that $a \in J_0$, $c \in J_n$, and for each $i < n$, $J_i + K_i$ and $J_{i+1} +K_i$ are $B$-indiscernible and have the same Ehrenfeucht-Mostowski type as $I$.
\end{cor}
\begin{proof}
  This is immediate from \cref{prop:weak-application-I} and \cite[Prop.~4.3]{BoundedUltra}.
\end{proof}

So in other words, under the assumptions of \cref{prop:weak-application-I}, there are indiscernible sequences that universally witness Lascar strong type in a strong way.\footnote{Ordinarily, $a$ and $a'$ having the same Lascar strong type over $B$ is witnessed by some sequence of indiscernible sequences $I_0,\dots,I_{n-1}$ with $a \in I_0$, $a' \in I_{n-1}$, and $I_i \cap I_{i+1} \neq \varnothing$ for all $i < n$, with no restriction on the Ehrenfeucht-Mostowski types of the $I_i$'s and no additional information of the form of $I_i \cap I_{i+1}$.} It seems plausible that \cref{prop:weak-application-I} and \cref{cor:weak-application-II} can be weakened to the assumption of the existence of a $k$-subtle cardinal, possibly by using something like \cite[Lem.~6]{Friedman_2001}.

\section{A family of examples}
\label{sec:examples}

One of the basic examples of a stable theory is theory of an equivalence relation. One way to understand the fact that this theory is stable is that it is `morally' a unary theory:\footnote{Although note that \cite{FOP} takes the perspective that stable theories are morally binary (and more generally that $\NFOP_k$ theories are morally $(k+1)$-ary). This difference has to do with how one conceptualizes the role of parameters. For the purposes of this paper, however, the use of Fodor's lemma in \cref{prop:indisc-extract-sloppy} clearly has a unary character to it.} The type of an element is uniquely determined by its `color' (i.e., the equivalence class that it is in). The only reason that the theory isn't literally unary is that the number of colors can `grow with the model.'

We can use this point of view to define analogous theories of higher arity, yielding examples of theories with bounded $k$-splitting that are neither stable nor $k$-ary. The idea is to simply consider a colored complete $k$-ary hypergraph but with a number of colors that can also `grow with the model,' which we implement with an equivalence relation on unordered $k$-tuples. This is similar to the example discussed in \cite[Prop.~3.29]{FOP} (which is a generic linear order on $k$-tuples from $k$ distinct sorts).

\begin{defn}
  Fix $k \geq 2$. Let $\Lc^{\Equiv}_k$ be a language with a $2k$-ary relation symbol $E(\xbar,\ybar)$. Let $\Equiv^{0}_k$ be the $\Lc^{\Equiv}_k$-theory that states
  \begin{itemize}
  \item $E$ is a partial equivalence relation on $k$-tuples,
  \item $E(\xbar,\xbar)$ if and only if $x_i \neq x_j$ for all $i < j < k$, and
  \item if $x_i \neq x_j$ for all $i < j < k$, then $E(\xbar,x_{\sigma(0)}\dots x_{\sigma(k-1)})$ for every permutation $\sigma : k \to k$.
  \end{itemize}
\end{defn}

\begin{prop}
  The finite models of $\Equiv^{0}_k$ form a \Fraisse\ class.
\end{prop}
\begin{proof}
  It is straightforward to check that models of $\Equiv^{0}_k$ have free amalgamation. The other properties of a \Fraisse\ class are obvious.
\end{proof}

Let $\Equiv_k$ be the theory of the \Fraisse\ limit of the class of finite models of $\Equiv^{0}_k$. Note that $\Equiv_k$ is $\aleph_0$-categorical and has quantifier elimination.

\begin{lem}\label{lem:T-k-small-trivial}
  For all $\ell < 2k$, two $\ell$-tuples $\abar$ and $\bbar$ in a model of $\Equiv_k$ have the same type if and only if they have the same type in the equational reduct.
\end{lem}
\begin{proof}
  The axioms of $\Equiv^{0}_k$ imply that there can be no realizations of the relation $E$ in an $\ell$-tuple for $\ell < 2k$. The result now follows from quantifier elimination.
\end{proof}

\begin{prop}
  $\Equiv_k$ is $2k$-ary but not $(2k-1)$-ary.
\end{prop}
\begin{proof}
  This follows immediately from quantifier elimination and \cref{lem:T-k-small-trivial}.
\end{proof}

Given the story we told to motivate the definition of $\Equiv_k$, one would expect that it should be `$k$-arily stable,' and while it does turn out to be $\NFOP_k$ (\cref{prop:Equiv-NFOP-k}), it also does not have bounded $k$-splitting.

\begin{prop}\label{prop:T-k-unbounded}
  $\Equiv_k$ has unbounded $(2k-2)$-splitting.
\end{prop}
\begin{proof}
  Fix an ordinal $\alpha$. Consider the $\Lc^{\Equiv}_k$-structure consisting of elements $(b_i)_{i<\alpha}$, $(c_i)_{i < \alpha}$, $(b'_i)_{i < \alpha}$, $(c'_i)_{i< \alpha}$, and $(a_i)_{i < 2k-2}$ satisfying that the only instances of the relation $E$ are those generated by $E(a_0\dots a_{k-2}b_i,\allowbreak a_{k-1}\dots a_{2k-3}c_i)$ for $i< \alpha$ and the axioms of $\Equiv_k$. (It is straightforward to check that such a structure exists.) Consider it as a set of parameters in a model of $\Equiv_k$. We need to argue that for every $\beta < \alpha$, the partitioned $(2k-2)$-type $\tp(a_0;\dots;a_{2k-3}/b_{<\alpha}c_{<\alpha}b'_{<\alpha}c'_{<\alpha})$ $(2k-2)$-splits over $b_{<\beta}c_{<\beta}b'_{<\beta}c'_{<\beta}$. It is immediate from quantifier elimination that for each $i < 2k-2$, \(b_\beta c_\beta \equiv_{b_{<\beta}c_{<\beta}b'_{<\alpha}c'_{<\beta}a_0\dots \hat{c}_i \dots a_{2k-3}} b'_\beta c'_\beta.\) Therefore since $E(a_0\dots a_{k-2}b_\beta,a_{k-1}\dots a_{2k-3}c_\beta)$ and $\neg E(a_0\dots a_{k-2}b'_\beta,a_{k-1}\dots a_{2k-3}c'_\beta)$ hold, we have an instance of $(2k-2)$-splitting.

  Since we can do this for any ordinal $\alpha$, $\Equiv_k$ has unbounded $(2k-2)$-splitting.
\end{proof}

\begin{prop}
  $\Equiv_k$ has $\aleph_0$-bounded $(2k-1)$-splitting.
\end{prop}
\begin{proof}
  Fix a set of parameters $B$ in a model of $\Equiv_k$ and a $(2k-1)$-partitioned type $\tp(\abar_0;\dots;\abar_{2k-2}/B)$. For each $k$-tuple $\dbar = d_0\dots d_{k-1}$ of elements from the tuple $\abar_0\abar_1\dots \abar_{2k-2}$, if there is a $k$-tuple $\cbar \in B\abar_0\dots \abar_{k-1}$ not entirely contained in $\abar_0\dots\abar_{k-1}$ such that $E(\dbar,\cbar)$ holds, let $\cbar_{\dbar}$ be some such $k$-tuple. Let $C$ be the set of elements of $B$ occurring in the tuples $\cbar_{\dbar}$. Clearly $|C| < \aleph_0$. We need to show that $\tp(\abar_0;\dots;\abar_{2k-2}/B)$ does not $(2k-1)$-split over $C$.

  Fix $\bbar =b_0\dots b_{\ell-1}$ and $\dbar = d_0\dots d_{\ell-1}$ in $B$ satisfying that $\bbar \equiv_{C\abar_0\dots \hat{\abar}_i \dots \abar_{2k-2}} \dbar$ for each $i < 2k-1$. By quantifier elimination, all we need to do is check that whenever there is an instance of the $E$ relation involving elements of $\bbar$, then the same relation holds with the corresponding elements of $\dbar$. Any instance of the $E$ relation that does not involve an element from each of the tuples $\abar_i$ for $i < 2k-1$ must be the same for $\bbar$ and $\dbar$ by our assumption, so the only cases we need to check are those that involve elements from each $\abar_i$.

  Assume that $E(\tbar,\sbar)$ holds, where every element of $\tbar$ and $\sbar$ is either in $C$, $\abar_i$ for some $i < 2k-1$, or $\bbar$ and some element is from $\bbar$. Since we need an element from each of the $\abar_i$'s to occur in either $\tbar$ or $\sbar$, there can be only one element of $\tbar\sbar$ that is not from $\abar_i$ for some $i < 2k-1$. Therefore no elements of $C$ can be involved, and, by symmetry, we may assume that the relation is of the form $E(e_0\dots e_{k-1},e_k\dots e_{2k-3} b_m$ for some $m < \ell$, where for each $j < 2k-2$, $e_j \in \abar_i$ for some $i < 2k-1$. By construction, we must have placed some $c \in C$ satisfying that $E(e_0\dots e_{k-1},e_k \dots e_{2k-3} c)$, so now we have that $E(e_k\dots e_{2k-3} c , e_k\dots e_{2k-3} b_m)$. By our type condition, we also have that $E(e_k\dots e_{2k-3} c, e_k\dots e_{2k-3} d_m)$, whereby $E(e_0\dots e_{k-1},e_k\dots e_{2k-3}d_m)$.

  We have established that whenever an instance of the $E$ relation occurs involving elements of $\bbar$, the same relation occurs involving the corresponding elements of $\dbar$. By symmetry the reverse is also true, so by quantifier elimination we have that $\bbar \equiv_{C \abar_0\dots \abar_{2k-2}} \dbar$, as required.
\end{proof}

Finally, it is natural to wonder about the status of $\FOP_k$ and $\IP_k$ in $\Equiv_k$. Gabriel Conant has provided us with the following adaptation to the case of stable formulas of \cite[Lem.~3.26]{FOP} and its proof. 

\begin{lem}\label{lem:Gabe}
  Let $T$ be a complete theory. Suppose $\varphi(\xbar;\ybar)$ is stable in $T$ (as a partitioned formula), with $|\xbar| = |\ybar| = k$. Let $z_0,\dots,z_{k}$ be a partition of the $2k$ free variables in $\varphi(\xbar;\ybar)$. Then $\theta(z_0;\dots;z_k)\coloneq \varphi(\xbar,\ybar)$ is $\NFOP_k$ in $T$.
\end{lem}
\begin{proof}
  We will prove the contrapositive. Assume that $\theta(z_0;\dots;z_k)$ has $\FOP_k$ in $T$. We will show that the formula $\varphi(\xbar;\ybar)$ is unstable in $T$.

  Just as in the proof of \cite[Lem.~3.26]{FOP}, there must be $s,t < k+1$ such that $z_s \subseteq \xbar$ and $z_t \subseteq \ybar$. By switching $\xbar$ and $\ybar$ if necessary, we may assume without loss of generality that $s < t$. There are now two cases.

  \vspace{0.5em}

  \noindent \emph{Case 1.} $t = k$.

  For $i < \omega$, let $\tau_i = 0^{s-1}\concat i \concat 0^{k-s} \in \omega^k$. Choose a linear order $<_\ast$ on $\omega^k \cup \omega$ such that for $j \in \omega$, $\tau_i <_\ast j$ if and only if $i < j$. By \cite[Prop.~2.3]{FOP}, there is an array $(a_i^j)_{i<\omega,j\leq k}$ such that $\theta(a^0_{i_0};\dots ; a_{i_k}^k)$ holds if and only if $(i_0,\dots,i_{k-1}) <_\ast i_k$. For $r \in (k-1)\setminus \{s\}$, let $c_r$ be the subtuple of $a_0^r$ corresponding to the overlap between $z_r$ and $\xbar$, and let $d_r$ be the subtuple of $a_0^r$ corresponding to the overlap between $z_r$ and $\ybar$. For $i < \omega$, let $\gbar_i$ be the enumeration of $(c_0,\dots,c_{s-1},a_i^{s},c_{s+1},\dots,c_{k-1},\varnothing)$ in the order corresponding to $\xbar$, and let $\hbarr_i$ be the enumeration of $(d_0,\dots,d_{s-1},\varnothing,d_{s+1},\dots,d_{k-1},a^k_j)$ in the order corresponding to $\ybar$. By construction, $\varphi(\gbar_i,\hbarr_j)$ is the same thing as $\theta(a_0^0;\dots;a_0^{s-1};a_i^s;a_0^{s+1};\dots;a_0^{k-1},a_0^k)$. Therefore we have that $\varphi(\gbar_i,\hbarr_j)$ holds if and only if $\tau_i <_\ast j$ and so if and only if $i < j$, whereby $\varphi(\xbar;\ybar)$ is unstable in $T$.

  \vspace{0.5em}

  \noindent \emph{Case 2.} $t < k$.

  For $i,j < \omega$, let $\tau_{i,j} = 0^{s-1}\concat i \concat 0^{t-s-1}\concat j \concat 0^{k-1} \in \omega^k$. Choose a linear order $<_\ast$ on $\omega^k \cup \omega$ such that for $i,j < \omega$, $\tau_{i,j} <_\ast 0$ if and only if $i<j$. By \cite[Prop.~2.3]{FOP}, there is an array $(a_i^j)_{i<\omega,j\leq k}$ such that $\theta(a^0_{i_0};\dots ; a_{i_k}^k)$ holds if and only if $(i_0,\dots,i_{k-1}) <_\ast i_k$. For $r \in (k+1)\setminus \{s,t\}$, let $c_r$ be the subtuple of $a_0^r$ corresponding to the overlap between $z_r$ and $\xbar$, and let $d_r$ be the subtuple of $a_0^r$ corresponding to the overlap between $z_r$ and $\ybar$. For $i < \omega$, let $\gbar_i$ be the enumeration of $(c_0,\dots,c_{s-1},a_i^s,c_{s+1},\dots,c_{t-1},\varnothing,c_{t+1},\dots,c_{k})$ in the order corresponding to $\xbar$, and let $\hbarr_i$ be the enumeration of $(d_0,\dots,d_{s-1},\varnothing,d_{s+1},\dots,d_{t-1},a_i^t,d_{t+1},\dots,d_k)$ in the order corresponding to $\ybar$. By construction $\varphi(\gbar_i,\hbarr_j)$ is the same thing as $\theta(a_0^0;\dots;a_0^{s-1};a_i^{s};a_0^{s+1},\dots,a_0^{t-1},a_j^t;a_0^{t+1},\dots,a_0^k)$. Therefore we have that $\varphi(\gbar_i,\hbarr_j)$ holds if and only if $\tau_{i,j} <_\ast 0$ and so if and only if $i < j$, whereby $\varphi(\xbar;\ybar)$ is unstable in $T$.
\end{proof}

\begin{prop}\label{prop:Equiv-NFOP-k}
  $\Equiv_k$ is $\NFOP_k$ and has IP$_{k-1}$.
\end{prop}
\begin{proof}
  It is immediate from the fact that $E$ is an equivalence relation on $k$-tuples that the partitioned formula $E(\xbar;\ybar)$ (with $|\xbar|=|\ybar| = k$) is stable in $\Equiv_k$. Therefore, by quantifier elimination, \cref{lem:Gabe}, and \cite[Thm.~1.5]{FOP}, $\Equiv_k$ is $\NFOP_k$.

 To see that $\Equiv_k$ has IP$_{k-1}$ fix a $k$-tuple $\abar$ and an array $(b_{i}^j)_{i<\omega,j<k-1}$ of distinct elements. For any $X \subseteq \omega^{k-1}$, we can find a $c$ such that for any $(i_0,\dots,i_{k-2}) \in \omega^{k-1}$, $E(\abar;b_{i_0}^0\dots b_{i_{k-2}}^{k-2}c)$ if and only if $(i_0,\dots,i_{k-2}) \in X$ (since any such pattern occurs in a model of $\Equiv_k^0$ and the class of models of $\Equiv_k^0$ has free amalgamation).
\end{proof}

So we've established that for every $k > 1$, the theory $\Equiv_k$ has IP$_{k-1}$, is $\NFOP_k$, has unbounded $(2k-2)$-splitting and bounded $(2k-1)$-splitting, and is $2k$-ary but not $(2k-1)$-ary, giving a simultaneous separation of $\NFOP_k$, bounded $k$-splitting, and $k$-arity.

Since $\Equiv_k$ is still $2k$-ary, we know that \cref{thm:bounded-k-splitting-indisc-extraction} is not entirely optimal for it, at least in some sense. \cref{thm:bounded-k-splitting-indisc-extraction} establishes that if $\mu$ is $k$-ineffable, then $\mu \to (\mu)_{\Equiv_k,n}$ for any $n < \omega$. Often in applications, merely finding an infinite indiscernible sequence (or an indiscernible sequence of length $\alpha$ for some relatively small countable ordinal $\alpha$) rather than an indiscernible sequence of some large cardinality is what really matters. Since $\Equiv_k$ is $2k$-ary, \cref{fact:Erdos-Rado-extraction} applies and implies that $\beth_{2k}^+ \to ((2^{\aleph_0})^+)_{\Equiv_k,n}$ for any $n < \omega$.

\section{Relationship to other higher-arity notions}
\label{sec:relationship-to-other}

In order to compare bounded $k$-splitting to $\NFOP_k$, we'll first recall some notation and facts from \cite[Sec.~3.2]{FOP}. $\Lc_k = \{P_0,\dots,P_{k},<,<_k,R\}$ is a language with $P_0,\dots,P_{k}$ unary relations, $<$ a binary relation, $<_k$ a $2k$-ary relation, and $R$ a $(k+1)$-ary relation. This language is used in \cite{FOP} for the indexing structure of a kind of generalized indiscernible that characterizes $\FOP_k$. $T_k$ is the universal $\Lc_k$-theory that says
\begin{itemize}
\item $P_0,\dots,P_{k}$ is a partition, 
\item $<$ is a linear order satisfying $P_0 < \dots < P_{k}$,
\item if $R(x_0,\dots,x_{k})$ holds, then $P_i(x_i)$ holds for each $i \leq k$,
\item $<_k$ only holds on $(P_0\times \dots \times P_{k-1})^2$ and is a linear order (i.e., a linear order on $k$-tuples in $P_0\times \dots \times P_{k-1}$), and
\item if $\xbar \leq_k \ybar$, $R(\ybar,w)$, and $w \leq z$, then $R(\xbar,z)$.\footnote{Note that the assumptions of this imply that $w$ and $z$ are in $P_k$.}
\end{itemize}
Note that $T_k$ imposes no restriction on the relationship between the relation $<$ on $P_0,\dots,P_{k-1}$ and the relations $<_k$ and $R$.

In \cite[Cor.~3.15]{FOP}, it is shown that the finite models of $T_k$ form a \Fraisse\ class. We let $\Hc_k$ denote its \Fraisse\ limit. Recall that given a structure $M$, a map $f : \Hc_k \to M^{<\omega}$ (where $M^{<\omega}$ is the set of finite tuples of elements of $M$) is an \emph{$\Hc_k$-indexed indiscernible sequence} if for any finite tuple $a_0,\dots,a_{k-1} \in \Hc_k$, the sorts of $f(a_0),\dots,f(a_{k-1})$ and $\tp(f(a_0),\dots, f(a_{k-1}))$ only depend on the quantifier-free type of $a_0\dots a_{k-1}$ in $\Hc_k$.

\begin{fact}[{\cite[Prop.~3.20]{FOP}}]\label{fact:FOP-k-indisc-witness}
  $T$ has $\FOP_k$ if and only if there is a model $M \models T$, a formula $\varphi(x_0,\dots,x_{k})$, and an $\Hc_k$-indexed indiscernible sequence $f : \Hc_k \to M^{<\omega}$ such that for any $(h_0,\dots,h_{k}) \in P_0(\Hc_k)\times \dots \times P_{k}(\Hc_k)$, we have that $f(h_i)$ is of the same sort as $x_i$ for each $i \leq k$ and $\varphi(f(h_0),\dots,f(h_k))$ holds if and only if $\Hc \models R(h_0,\dots,h_k)$.
\end{fact}

\begin{prop}\label{prop:if-bounded-k-then-NFOP-k}
  If $T$ has bounded $k$-splitting, then $T$ is $\NFOP_k$.
\end{prop}
\begin{proof}
  We will prove the contrapositive. Assume that $T$ has $\FOP_k$. Find $\varphi(x_0,\dots,x_{k})$, $M \models T$, and $f : \Hc_k \to M^{<\omega}$ as in the statement of \cref{fact:FOP-k-indisc-witness}.

  Fix $\ell < \omega$, and let $K$ consider the following finite model of $T_k$:
  \begin{itemize}
  \item There are elements $b_0 < c_0 < b_1 < c_2 < \dots < b_{\ell -1} < c_{\ell-1}$ in $P_0$. 
  \item For each $i \leq k$ with $i > 0$, there is a single element $a_i$ in $P_i$.
  \item $(b_0,a_1,\dots,a_{k-1})  <_k (b_1,a_1,\dots,a_{k-1}) <_k \dots <_k (b_{\ell -1},a_1,\dots,a_{k-1}) < _k (c_{\ell -1},a_1,\dots,a_{k-1}) <_k \dots <_k (c_1,a_1,\dots,a_{k-1}) <_k (c_0,a_1,\dots,a_{k-1})$. 
  \item For each $j < \ell$, $R(b_j,a_1,\dots,a_{k})$ and $\neg R(c_j,a_1,\dots,a_{k})$ hold. 
  \end{itemize}
    Note that by construction, we have that for each $i < k$ and $j < \ell$, $b_j$ and $c_j$ have the same quantifier-free type over $b_{<j}c_{<j}a_0\dots \hat{a}_i \dots a_{k-1}$.

  \begin{sloppypar}
    We may regard $K$ as an $\Lc_k$-substructure of $\Hc_k$. Note that since $f$ is an $\Hc_k$-indexed indiscernible sequence, we have that for each $i < k$ and $j < \ell$, $f(b_j) \equiv_{f(b_{<j}c_{<j}a_0\dots \hat{a} \dots a_{k-1})} f(c_j)$. Since $\varphi(f(a_0),\dots,f(a_{k-1}),f(b_j))$ and $\neg \varphi(f(a_0),\dots,f(a_{k-1}),f(c_j))$ hold, we have that $\tp(f(a_0);\dots;f(a_{k-1})/f(b_{\leq j}c_{\leq j}))$ $k$-splits over $f(b_{<j}c_{<j})$. By compactness, we can stretch this to arbitrary length, implying that $T$ has unbounded $k$-splitting. \qedhere
  \end{sloppypar}
\end{proof}

At the moment, there isn't a known analog of \cref{fact:FOP-k-indisc-witness} for $\OP_2$, which raises the following question.

\begin{quest}
  Does bounded $2$-splitting imply $\NOP_2$?
\end{quest}

The converse of \cref{prop:if-bounded-k-then-NFOP-k} fails strongly. It turns out that even an $\NIP$ theory can have unbounded $k$-splitting for every $k$. 
Recall that $\kappa(\omega)$ is the smallest $\omega$-Erd\H{o}s cardinal. By convention, let $\kappa(\omega) = \infty$ if no such cardinal exists. %

\begin{fact}[{Kaplan, Shelah \cite{Kaplan_2014}}]\label{fact:NIP-no-indisc}
  $(\ZFC)$ There is a countable $\NIP$ theory $T_{2^{<\omega}}$ such that for every $\alpha < \kappa(\omega)$, there is a sequence $(a_i)_{i < \alpha}$ in a model of $T_{2^{<\omega}}$ containing no infinite indiscernible subsequence.
\end{fact}

\begin{prop}\label{prop:k-ineff-not-NFOP}
The theory $T_{2^{<\omega}}$ from \cite{Kaplan_2014} is $\NIP$ but does not have bounded $k$-splitting for any $k < \omega$.
\end{prop}
\begin{proof}
  $T_{2^{<\omega}}$ is $\NIP$ by \cref{fact:NIP-no-indisc}. Assume the existence of a $k$-inaccessible cardinal for each $k < \omega$ (which is a fairly small assumption in the grand scheme of the large cardinal hierarchy). By \cite[Lem.~13]{Friedman_2001}, the smallest $k$-ineffable cardinal is smaller than the smallest $\omega$-Erd\H{o}s cardinal (if it exists). Therefore if $T_{2^{<\omega}}$ has bounded $k$-splitting, we arrive at a contradiction given \cref{thm:bounded-k-splitting-indisc-extraction} and \cref{fact:NIP-no-indisc}. (For the dour large cardinal skeptic, we give a direct proof of unbounded $k$-splitting in $T_{2^{<\omega}}$ in \cref{sec:unbounded-k-splitting-in-example}.)
\end{proof}

In particular by \cite[Prop~2.8]{FOP}, this theory $T$ has $\NFOP_k$ for all $k > 1$, so this establishes that for any $\ell$ and $k > 1$, $\NFOP_k$ does not imply bounded $\ell$-splitting. Of course, the triangle-free random graph demonstrates that bounded $2$-splitting does not even imply NATP (and so does not imply $\NIP$, $\NTP_2$, or $\NSOP_1$).

This allows us to completely characterize implications between $k$-arity, bounded $k$-splitting, $\NFOP_k$, and $k$-dependence for various $k$ (see Figure~\ref{fig:more-implications}).

\begin{figure}
  \centering
\begin{tikzcd}
	{\text{Stability}} & {\text{NIP}} & {\text{NFOP}_2} & {\text{NIP}_2} & {\text{NFOP}_3} & {\text{NIP}_3} & \cdots \\
	&& \begin{array}{c} \text{Bounded} \\ 2\text{-splitting} \end{array} && \begin{array}{c} \text{Bounded} \\ 3\text{-splitting} \end{array} && \cdots \\
	{\text{Unarity}} && {\text{Binarity}} && {\text{Ternarity}} && \cdots
	\arrow[from=1-1, to=1-2]
	\arrow[from=1-1, to=2-3]
	\arrow[from=1-2, to=1-3]
	\arrow[from=1-3, to=1-4]
	\arrow[from=1-4, to=1-5]
	\arrow[from=1-5, to=1-6]
	\arrow[from=1-6, to=1-7]
	\arrow[from=2-3, to=1-3]
	\arrow[from=2-3, to=2-5]
	\arrow[from=2-5, to=1-5]
	\arrow[from=2-5, to=2-7]
	\arrow[from=3-1, to=1-1]
	\arrow[from=3-1, to=3-3]
	\arrow[from=3-3, to=2-3]
	\arrow[from=3-3, to=3-5]
	\arrow[from=3-5, to=2-5]
	\arrow[from=3-5, to=3-7]
\end{tikzcd}
  \caption{Relationship between $k$-arity, bounded $k$-splitting, $\NFOP_k$, and $k$-dependence. All implications are strict, and any implication not depicted is known to fail.}
  \label{fig:more-implications}
\end{figure}
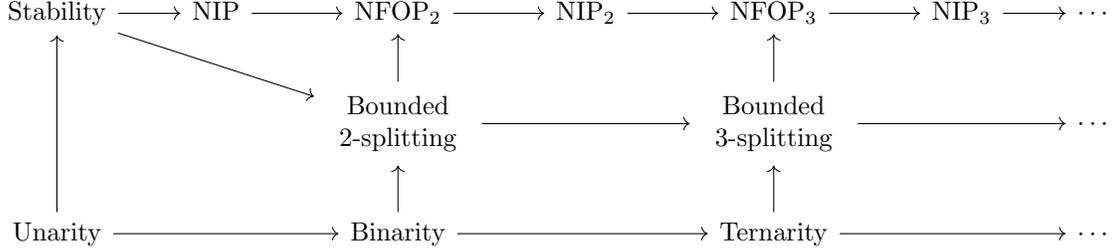

\begin{prop}\label{prop:fig-strict}
  The implications in Figure~\ref{fig:more-implications} are all strict and moreover there are no implications between the classes of theories shown beyond those depicted.
\end{prop}
\begin{proof}
  Strictness of the implications in the top row is established in \cite[Sec.~3.4]{FOP}. Strictness of the implications in the second row is established later in \cref{prop:hypergraph-example}, and the examples of the random $k$-ary hypergraph also establish strictness of the bottom row. \cref{prop:k-ineff-not-NFOP} establishes that there are no implications from the top row into the lower two rows, aside from the fact that stable theories have bounded $k$-splitting for any $k$. Finally, the existence of a stable theory that is not $k$-ary for any $k$ (such as $\mathsf{ACF}_0$) establishes that there are no implications from the top two rows into the bottom row.
\end{proof}



%

%

The relationship between bounded $2$-splitting, treelessness, and $\Cc$-lessness is also unclear.

\begin{quest}\label{quest:treeless-bounded-k-splitting}
  Does treelessness or $\Cc$-lessness imply bounded $k$-splitting for $k \geq 2$? Does bounded $2$-splitting imply treelessness or $\Cc$-lessness? 
\end{quest}

$T_{2^{<\omega}}$ in not a counterexample for \cref{quest:treeless-bounded-k-splitting}.

\begin{prop}
  $T_{2^{<\omega}}$ is not treeless.
\end{prop}
\begin{proof}
  (See \cite{Kaplan_2024,Kaplan_2014} for the relevant notation.) We can embed a copy of the tree $\omega^{\leq \omega^2}$ as $(a_{\tau})_{\tau \in \omega^{\leq \omega^2}}$ into $P_{\varnothing}$ in a model of $T_{2^{<\omega}}$ with its natural successor, limit, and meet structure. Let $R \subseteq \omega^{\leq\omega^2}$ be the subset consisting of elements whose height is either $0$ or a limit ordinal. We may regard $R$ as an $\Lc_{0,P}$-structure by interpreting $\wedge$ as the meet, $\tleq$ as the extension order, $<_{\lex}$ as the lexicographic order, and $P$ as the set of elements of height ${\omega^2}$. By quantifier elimination, we now have that $(a_{\tau})_{\tau \in R}$ is a treetop indiscernible which clearly is not an indiscernible sequence in the lexicographic ordering (even over $\varnothing$). Therefore $T_{2^{<\omega}}$ is not treeless.
\end{proof}

There are also higher-arity notions of simplicity. The \emph{$n$-simplicity} defined by Kim, Kolesnikov, and Tsuboi in \cite{Kim_2008} seems like it might be relevant, as it has to do with higher-arity amalgamation problems, but it is not implied by stability. For the other direction, it is relatively easy (as discussed in \cite{Kim_2008}) to show that the theory of a random $k$-ary hypergraph is $n$-simple for all $n$, but we also have the following.

\begin{prop}\label{prop:hypergraph-example}
  For any $k > 1$, the random $k$-ary hypergraph has unbounded $\ell$-splitting for all $\ell < k$.
\end{prop}
\begin{proof}
  Let $T$ be the theory of the random $k$-ary hypergraph. It is sufficient to show that $T$ has unbounded $(k-1)$-splitting by \cref{prop:next-k}. Fix an ordinal $\alpha$. Fix a $(k-1)$-partitioned tuple $a_0;\dots;a_{k-2}$ of distinct elements and two sequences $(b_i)_{i<\alpha}$ and $(b'_i)_{i<\alpha}$ such that the only edges are of the form $\{a_0,\dots,a_{k-2},b_i\}$. This can be witnessed in a model of $T$ and clearly witnesses unbounded $(k-1)$-splitting.
\end{proof}

In hindsight this mismatch doesn't seem too surprising. The difference between stability and simplicity has to do with uniqueness. Simplicity allows for amalgamation of sufficiently independent configurations, but doesn't guarantee uniqueness of these amalgamations. The analog of $3$-amalgamation (i.e., the independence theorem) for stability is the following: If $B \ind_A C$, $d \ind_A B$, $e \ind_A C$, and $d \equiv^{\mathrm{L}}_{A} e$, then there is a \emph{unique} type $p(x) \in S_x(ABC)$ extending $\tp(d/AB)\cup \tp(e/AB)$ that does not fork over $A$. The issue of uniqueness of higher amalgamation problems was originally studied by Hrushovski in the context of stable theories in \cite{HRUSHOVSKI_2012}. In particular, he introduced the notion of \emph{$n$-uniqueness}. Uniqueness of non-forking extensions is $2$-uniqueness, so the natural $k$-ary analog in the sense of this paper is $(k+1)$-uniqueness. This raises a natural question.

\begin{quest}
  Is there a characterization of bounded $k$-splitting in terms of higher-arity amalgamation with some kind of uniqueness condition, such as $(k+1)$-uniqueness? 
\end{quest}

Chernikov has shown that in simple theories with $4$-amalgamation, the conditions of $3$-uniqueness, $2$-dependence, NOP$_2$, NFOP$_2$, and $\Cc$-lessness are equivalent \cite{ChernikovUpcoming2025,ChernikovOberwolfach2023}.

 \begin{quest}
   In simple theories with $4$-amalgamation, are the above conditions equivalent to bounded $2$-splitting?
 \end{quest}

 One thing to note though with regards to this question is that $k$-simplicity, $k$-amalgamation, and $k$-uniqueness are of a different character than $k$-dependence or the various notions of $k$-ary stability. $k$-dependence and $k$-ary stability become weaker as one passes to larger $k$, but $k$-simplicity, $k$-amalgamation, and $k$-uniqueness become stronger. In particular, all $k$-simple theories are simple.


\section{Some speculation}
\label{sec:speculation}

It is natural to ask is whether more of the machinery of stability theory can be generalized to the context of theories with bounded $k$-splitting. Obviously the central notions of stability theory are dividing and forking, and one of the central results is the symmetry of forking. Stability can easily be shown to satisfy the following weak symmetry property: For any model $M$ and parameters $\abar$ and $\bbar$, if $\bbar \indu_M \abar$, then $\tp(\abar/M\bbar)$ does not split over $M$.\footnote{Recall that $\bbar \indu_M \abar$ means that $\tp(\bbar/ M \abar)$ is finitely satisfiable in $M$.} We can directly get a similar result.

\begin{prop}\label{prop:weak-symmetry}
If $T$ has bounded $k$-splitting, then for any $M \models T$ and $\abar_0,\dots,\abar_{k-1},\bbar$, if $\bbar \indu_M \abar_0\dots \abar_{k-1}$, then $\tp(\abar_0;\dots;\abar_{k-1}/M\bbar)$ does not $k$-split over $M$.
\end{prop}
\begin{proof}
  Assume that $\bbar \indu_M \abar_0\dots \abar_{k-1}$ and $\tp(\abar_0;\dots;\abar_{k-1}/M\bbar)$ does $k$-split over $M$. Let $b,b' \in \bbar$ be elements that witness that $\tp(\abar_0;\dots;\abar_{k-1}/M\bbar)$ $k$-splits over $M$. In particular, note that $b \equiv_{M\abar_0\dots \hat{\abar}_i \dots \abar_{k-1}} b'$ for each $i < k$.

  Fix a global $M$-coheir $p(\xbar)$ such that $\bbar \models p \res M \abar_0\dots \abar_{k-1}$. Fix an ordinal $\alpha$. Build a Morley sequence $(\bbar_j)_{j \in \alpha^\ast}$ (where $\alpha^\ast$ is $\alpha$ with the reversed order) in $p$ over $M\bbar$. Let $(b_i,b'_j)_{j \in \alpha^\ast}$ be the elements of $\bbar_j$ corresponding to $b$ and $b'$. We need to argue that for each $j \in \alpha^\ast$ and each $i < k$, $b_j \equiv_{M\abar_0\dots \hat{\abar}_i\dots \abar_{k-1}\bbar_0\dots \bbar_{j-1}} b'_j$. For each $i < k$, since $b_j \equiv_{M \abar_0\dots \hat{\abar}_i\dots \abar_{k-1}}b'_j$, we have that $\abar_0 \dots \hat{\abar}_i \dots \abar_{k-1} b_j \equiv_M \abar_0 \dots \hat{\abar}_i \dots \abar_{k-1} b'_j$, so what we want follows from the fact that $\bbar_0\dots \bbar_{k-1} \indu_M \abar_0\dots \abar_{k-1}b_j b'_j$. Since we can do this for any $\alpha$, $T$ has unbounded $k$-splitting.
\end{proof}

As mentioned in the introduction, Shelah's original definition of forking was in terms of strong splitting. There are a few ways one might try to adapt this definition to $k$-splitting, for instance, one could try the following non-dividing-like condition on a $k$-partitioned type $\tp(a_0;\dots;a_{k-1}/B)$ with some smaller set of parameters $C \subseteq B$:
\begin{itemize}
\item[$(\heartsuit)$] For every $b \in B$ and sequence $(b_i)_{i< \omega}$ with $b_0 = b$, if $(b_i)_{i < \omega}$ is $Ca_0\dots \hat{a}_i \dots a_{k-1}$-indiscernible for each $i < k$, then there is a $Ca_0\dots a_{k-1}$-indiscernible sequence $(b'_i)_{i <\omega}$ with $b'_0 = b$ such that $b_{<\omega} \equiv_{Ca_0\dots \hat{a}_i\dots a_{k-1}} b'_{<\omega}$ for each $i < k$.
\end{itemize}
And then one could define an analogous non-forking-like condition in terms of global extensions satisfying $(\heartsuit)$. This kind of higher amalgamation of indiscernible sequences is reminiscent of the $k$-distality introduced by Walker in \cite{WALKER_2022}, which is itself related to $k$-dependence. That said, the relationship between the two concepts shouldn't necessarily be that straightforward, as stable theories and distal theories are in some sense polar opposites among $\NIP$ theories.

A more compelling symmetry result than \cref{prop:weak-symmetry} would be something like this, possibly with some additional conditions:
\begin{quote}
  \begin{sloppypar}
    If $T$ has bounded $k$-splitting and $\tp(a_0;\dots;a_{k-1}/Bc)$ does not \underline{\hphantom{$k$-arily fork}} over $B$, then $\tp(c;a_1;\dots;a_{k-1}/Ba_{0})$ does not \underline{\hphantom{$k$-arily fork}} over $B$.
  \end{sloppypar}
\end{quote}
This would mean that this mystery condition on $\tp(a_0;\dots;a_{k-1}/Bc)$ does actually represent some notion of independence over $B$ for the set $\{a_0,\dots,a_{k-1},c\}$. At the moment, however, it is unclear if any advanced structural theory like this is possible. Moreover, the use of generalized indiscernible sequences in \cite{FOP} (and our use of this machinery in \cref{prop:if-bounded-k-then-NFOP-k}) suggests that it may be necessary to go beyond the sole consideration of indiscernible sequences in order to develop such a structure theory.

\appendix

\section{\texorpdfstring{Treelessness implies $\NOP_2$ (and therefore $\NFOP_2$)}{Treelessness implies NOP\texttwoinferior\ (and therefore NFOP\texttwoinferior)}}
\label{sec:treelessness-implies-NFOP}

Here we will give a proof that treeless theories are $\NOP_2$ (sharpening \cite[Prop.~3.21]{Kaplan_2024}) which was worked out by the author and Gabriel Conant during his visit to the University of Maryland in 2023. Chernikov has proved the analogous implication for $\Cc$-lessness in \cite{ChernikovUpcoming2025}.

Recall that a theory $T$ has the \emph{$2$-order property} or $\OP_2$ if there is a formula $\varphi(x;y,z)$ and sequences $(b_i)_{i \in \Qb}$ and $(c_j)_{j \in \Qb}$ such that for any non-decreasing function $f : \Qb \to \Qb$, there is an $a_f$ such that $\varphi(a_f;b_i,c_j)$ holds if and only if $j \leq f(i)$ \cite{Takeuchi2OrderProperty}. By a fairly easy compactness argument, if $T$ has $\OP_2$, then it has the same condition with any linear order in place of $\Qb$.

\begin{defn}
  Given a linear order $L$, say that a set $X \subseteq L^2$ is \emph{downright closed} if for any $(i,j) \in X$ and $(k,\ell) \in L^2$ with $k \geq i$ and $\ell \leq j$, $(k,\ell) \in X$.
\end{defn}

Note that for any non-increasing function $f : L \to L$, the set $\{(i,j) : j \leq f(i)\}$ is downright closed. Note also that any intersection of downright closed sets is downright closed.

\begin{lem}\label{lem:downright-conversion}
  If $T$ has the $2$-order property, then for any linear order $L$, we can find $(b_i)_{i \in L}$ and $(c_j)_{j \in L}$ such that for any downright closed set $X \subseteq L^2$, there is an $a_X$ such that $\varphi(a_X;b_i,c_j)$ holds if and only if $(i,j) \in X$.
\end{lem}
\begin{proof}
    Let $K$ be the completion of the linear order $2 \cdot L$ (i.e., $L$ with each point replaced by an increasing pair of points). We will regard $L$ as a subset of $K$ by identifying each $\ell \in L$ with $(1,\ell)$ (i.e., the larger of the two points associated to $\ell$).

    By compactness, we can find $(b_i)_{i \in K}$ and $(c_j)_{j \in K}$ such that for any non-decreasing function $f : K \to K$, there is an $a_f$ such that $\varphi(a_f;b_i,c_f)$ holds if and only if $j \leq f(i)$. 

    Fix a downright closed set $X \subseteq L^2$. Define a function $f : K \to K$ by $f_X(i) = \sup\{j \in L : (\exists \ell \in L)\ell \leq i\wedge (\ell,j) \in X\}$, where the supremum is computed in $K$ and therefore always exists. Note that $f_X$ is a non-decreasing function by construction.

    \vspace{0.5em}

    \begin{claim}{}
      For any $(i,j) \in L^2$, $j \leq f_X(i)$ if and only if $(i,j) \in X$.
    \end{claim}
    \begin{claimproof}{}
      First assume that $(i,j) \in X$. The clearly $j$ is in the set in the supremum that defines $f_X(i)$, so $j \leq f_X(i)$.

      Now assume that $j \leq f_X(i)$. By definition, this implies that $j \leq \sup\{m \in L : (\exists \ell \in L)\ell \leq i\wedge (\ell,m) \in X\}$. Since $X$ is downright closed as a subset of $L$, whenever $(\ell,m) \in X$ for some $\ell \leq i$, we also have that $(i,m) \in X$. Therefore $j \leq \sup\{m \in L : (i,m) \in X\}$. Since $K$ is not dense below $j$ (by our choice of embedding of $L$ into $K$), we must have that $j \in \{m : (i,m) \in X\}$, or in other words that $(i,j) \in X$.
    \end{claimproof}
    \vspace{0.5em}

    Now by restricting to $L$ as a subset of $K$, we can take $a_{f_X}$ for $a_X$ for any downright closed $X$ in order to get the required configuration.
\end{proof}

For the remainder of this proof we will use some notation and terminology from \cite{Kaplan_2024}. In particular, we will denote the lexicographical ordering on $\omega^{\leq \omega}$ by $<_{\lex}$ and the extension ordering on $\omega^{\leq \omega}$ by $\triangleleq$. Given $\alpha$ and $\beta$ in $\omega^{\leq \omega}$, we will write $\alpha\wedge \beta$ for the greatest common initial segment of $\alpha$ and $\beta$. A family $(a_\eta)_{\eta \in \omega^\omega}$ is a \emph{treetop indiscernible} if the family $(a_\eta)_{\eta \in \omega^{\leq \omega}}$ (where $a_\eta = \varnothing$ for $\eta \in \omega^{<\omega}$) is a generalized indiscernible family in the sense that for any tuple $\eta_0,\dots,\eta_{k-1} \in \omega^{\omega}$, $\tp(a_{\eta_0},\dots,a_{\eta_{k-1}})$ only depends on the quantifier-free type of $\eta_0,\dots,\eta_{k-1}$ in the language $\{\triangleleq,\wedge,<_{\lex},P\}$, where $P$ is a unary predicate selecting out the leaves of $\omega^{\leq \omega}$ (i.e., the elements of height $\omega$).

Consider the sets $(X_\eta)_{\eta \in \omega^\omega}$ in the proof of \cite[Prop.~3.21]{Kaplan_2024}:
\[
  X_\eta = \{(\nu,\xi) \in \omega^\omega\times \omega^\omega : \eta <_{\lex} \nu <_{\lex} \xi~\text{and}~\eta \wedge \nu \triangleless \nu \wedge \xi\}.
\]

\begin{lem}\label{lem:X-downright-closed}
  For each $\eta \in \omega^\omega$, there are downright closed (relative to $<_{\lex}$) sets $Y_\eta$ and $Z_\eta$ such that $X_\eta = Y_\eta \setminus Z_\eta$.
\end{lem}
\begin{proof}
  Define $Y_\eta$ and $Z_\eta$ as follows:
  \begin{align*}
    Y_\eta &= \{(\nu,\xi) \in \omega^\omega \times \omega^\omega : \eta <_{\lex} \nu~\text{and}~(\nu \wedge \eta \triangleless \nu \wedge \chi~\text{or}~\nu \geq_{\lex} \xi)\} \\
    Z_\eta &= \{(\nu,\xi) \in \omega^\omega \times \omega^\omega : \nu \geq_{\lex} \xi\}.
  \end{align*}
  Clearly $X_\eta = Y_\eta \setminus Z_\eta$. Furthermore, $Z_\eta$ is clearly downright close, so we just need to show that $Y_\eta$ is downright closed. Fix $(\nu,\xi) \in Y_\eta$ and assume that $\alpha \geq_{\lex} \nu$ and $\beta \leq_{\lex} \xi$. Since $\alpha \geq_{\lex} \nu$, we still have that $\eta <_{\lex} \alpha$. At this point there are two cases.
  \begin{enumerate}
  \item Assume $\alpha \geq_{\lex} \beta$. We have immediately that $(\alpha,\beta) \in Y_\eta$.
  \item Assume $\alpha <_{\lex} \beta$. It is a basic fact about trees that if $\nu \leq_{\lex} \alpha \leq_{\lex} \beta \leq_{\lex}\xi$, then $\nu \wedge \xi \triangleleq \alpha \wedge \beta$. Furthermore, if $\nu <_{\lex} \nu \leq_{\lex} \alpha \leq_{\lex} \xi$ and $\nu \wedge \nu \triangleless \nu \wedge \xi$, then $\eta \wedge \nu = \eta \wedge \alpha$. Therefore we have that
    \[
      \eta \wedge \alpha = \eta \wedge \nu \triangleless \nu \wedge \xi \triangleleq \alpha \wedge \beta,
    \]
    and so $\eta \wedge \alpha \triangleless \alpha \wedge \beta$. \qedhere
  \end{enumerate}
\end{proof}

\begin{prop}\label{prop:treeless-NOP}
  Any treeless theory $T$ does not have the $2$-order property.
\end{prop}
\begin{proof}
  We will prove the contrapositive. Let $T$ be a theory with the $2$-order property witnessed by the formula $\varphi(x;y,z)$. We would like to show that $T$ is `treeful' (i.e., not treeless). With the groundwork of our two lemmas, the proof is mostly the same as the proof of \cite[Prop.~3.21]{Kaplan_2024}.

  Consider $(\omega^\omega,<_{\lex})$ as a linear order. By \cref{lem:downright-conversion}, we can find $(b_\eta)_{\eta \in \omega^\omega}$ and $(c_\eta)_{\eta \in \omega^\omega}$ such that for any downright closed $X \subseteq \omega^\omega\times \omega^\omega$, there is an $a_X$ such that $\varphi(a_X;b_\eta,c_\xi)$ holds if and only if $(\eta,\xi) \in X$. 

  Let $(Y_\eta)_{\eta \in \omega^\omega}$ and $(Z_\eta)_{\eta\in \omega^\omega}$ be as in \cref{lem:X-downright-closed}. For each $\eta \in \omega^\omega$, let $d_\eta = a_{Y_\eta}$ and $e_\eta = a_{Z_\eta}$. Consider the treetop-indexed family $(d_\eta e_\eta b_\eta c_\eta)_{\eta \in \omega^\omega}$. Let $\psi(x,w;y,z) = \varphi(x;y,z) \wedge \neg \varphi(w;y,z)$. We now have that for each $\eta,\nu,\xi \in \omega^\omega$, $\psi(d_\eta,e_\eta;b_\nu,c_\xi)$ holds if and only if $(\nu,\xi) \in X_\eta$. Let $(e'_\eta d'_\eta b'_\eta c'_\eta)_{\eta \in \omega^\omega}$ be a treetop indiscernible locally based on $(d_\eta e_\eta b_\eta c_\eta)_{\eta \in \omega^\omega}$. By construction, it will still be the case that $\psi(e'_\eta,d'_\eta;b'_\nu,c'_\xi)$ holds if and only if $(\nu,\xi) \in X_\nu$ (i.e., if and only if $\eta <_{\lex} \nu <_{\lex} \xi$ and $\eta \wedge \nu \triangleless \nu \wedge \xi$). Any $\eta_0,\eta_1,\eta_2,\eta_3 \in \omega^\omega$ satisfying $\eta_0 <_{\lex} \eta_1 <_{\lex} \eta_2 <_{\lex} \eta_3$ and $\eta_0 \wedge \eta_1 \trianglegess \eta_1 \wedge \eta_3$ and $\eta_0 \wedge \eta_2 \triangleless \eta_2 \wedge \eta_3$ now witnesses that $(e'_\eta d'_\eta b'_\eta c'_\eta)_{\eta \in \omega^\omega}$ is not an indiscernible sequence, whereby $T$ is treeful.
\end{proof}

\begin{cor}
  If $T$ is treeless, then it is $\NFOP_2$.
\end{cor}
\begin{proof}
  This is immediate from \cref{prop:treeless-NOP} and the fact that $\NOP_2$ implies $\NFOP_2$.
\end{proof}

Of course it already followed from \cite[Prop.~3.21]{Kaplan_2024} and \cite[Prop.~2.8]{FOP} that any treeless theory is $\NFOP_k$ for any $k\geq 3$. Moreover, by \cite[Ex.~3.17]{Kaplan_2024}, we know that the inclusion of treeless theories into $\NOP_2$ theories is strict. As noted in \cite{FOP}, it is open whether the inclusion of $\NOP_2$ theories into $\NFOP_2$ theories is strict, but by \cite[Prop.~3.25]{FOP} this is the only implication in Figure~\ref{fig:implications} that is not known to be strict.

\section{Explicit unbounded $k$-splitting in the Kaplan-Shelah theory}
\label{sec:unbounded-k-splitting-in-example}

Here we will remove the large cardinal assumption from \cref{prop:k-ineff-not-NFOP}. See \cite[Sec.~3]{Kaplan_2014} for the relevant notation (except for concatenation of sequences, which we will denote by $\sigma \concat \tau$). We will be working with the specific case of $S = 2^{<\omega}$, although our argument makes it clear that we get unbounded $k$-splitting in $T_S$ whenever the indexing tree $S$ has height at least $k$, so even $S = (\omega, <)$ would work for our purposes here. It is likely that there is an even simpler example, but we have not pursued this. %

Fix $k < \omega$ and a limit ordinal $\alpha$. We will describe a model $M$ of $T_{2^{<\omega}}^{\forall}$ containing a $k$-splitting chain of length $\alpha$. Let each tree $(P_\eta,<_\eta)_{\eta \in 2^{<\omega}}$ in $M$ be a copy of $2^{<\omega\cdot \alpha+\omega}$ (ordered by extension). Everything important in the construction will be happening in the trees $(P_{0^n}, <_{0^n})_{n < k}$. Note that all of the $\Lc_{2^{<\omega}}$-structure of $M$ is now defined except for the functions $G_{\eta,\eta\concat i}$ for $\eta \in 2^{<\omega}$ and $i < 2$. The definitions of the $G_{\eta,\eta\concat i}$'s will be clearer if we name some specific elements of the above trees first before defining the $G_{\eta,\eta\concat i}$'s. Fix the following names:
\begin{itemize}
\item For each $i < k$, let $a_i = 0^{\omega\cdot \alpha}$ in $P_{0^i}$.
\item Let $a'_{k-1} = 1^{\omega\cdot \alpha}$ in $P_{0^{k-1}}$.
\item For each $i < k-1$ and $\beta < \alpha$, let $c_{i,\beta} = 0^{\omega \cdot \beta}\concat 1$ and let $b_{i,\beta} = c_{i,\beta} \concat 1^{\omega}$.
\end{itemize}
Note that $\mathrm{suc}_{0^i}(a_i\wedge_{0^i} b_{i,\beta},b_{i,\beta}) = c_{i,\beta}$ for each $i < k-1$ and $\beta < \alpha$. Now we need to give a complete definition of the $G_{\eta,\eta\concat i}$'s, but first we should specify the important values for later in the proof. What matters is the following:
\begin{itemize}
\item $G_{0^i,0^{i+1}}(c_{i,\beta}) = b_{i+1,\beta}$ for each $i < k-2$.
\item $G_{0^{k-2},0^{k-1}}(c_{k-2,\beta})$ is $a_{k-1}$ if $\beta$ is an even ordinal and is $a'_{k-1}$ if $\beta$ is an odd ordinal.
\end{itemize}
The rest of the definition is just chosen to satisfy the axioms of $T_{2^{<\omega}}^{\forall}$ and avoid spoiling the type equalities we need to establish $k$-splitting: %
\begin{itemize}
\item For each $i < k-2$ and $\beta < \alpha$, if $\sigma\in P_{0^i}$ is a finite extension of $c_{i,\beta}$, let $G_{0^i,0^{i+1}}(\sigma) = b_{i+1,\beta}$.
\item For even $\beta < \alpha$, if $\sigma$ is a finite extension of $c_{k-2,\beta}$, let $G_{0^{k-2},0^{k-1}}(\sigma) = a_{k-1}.$
\item For odd $\beta < \alpha$, if $\sigma$ is a finite extension of $c_{k-2,\beta}$, let $G_{0^{k-2},0^{k-1}}(\sigma) = a'_{k-1}.$
\item For any $\eta \in 2^{<\omega}$, $i < 2$, and $\sigma \in \mathrm{Suc}(P_\eta)$, if $G_{\eta,\eta\concat i}(\sigma)$ is not already defined, let it be $\varnothing$ in $P_{\eta \concat i}$.
\item For $\sigma \notin \bigcup_{\eta \in 2^{<\omega}}\mathrm{Suc}(P_\eta)$, $G_{\eta,\eta\concat i}(\sigma) = \sigma$ (as required by $T_{2^{<\omega}}^{\forall}$).
\end{itemize}
This completes the description of the model $M \models T_{2^{<\omega}}^{\forall}$.

Since $T_{2^{<\omega}}$ has quantifier elimination and is the model completion of $T_{2^{<\omega}}^{\forall}$ \cite[Cor.~3.20]{Kaplan_2014}, this structure $M$ uniquely specifies the type of a subset of a model of $T_{2^{<\omega}}$. Our goal is to show that for every even $\beta < \alpha$, $\tp(a_0;\dots;a_{k-1}/b_{0,\leq \beta+1})$ $k$-splits over $b_{0,<\beta}$, specifically as witnessed by the elements $b_{0,\beta}$ and $b_{0,\beta+1}$. The formula on which they differ is (relatively) straightforward:
\begin{align*}
  f_0(x_0,y) &\coloneq G_{\varnothing,0}(\mathrm{suc}_{\varnothing}(x_0 \wedge_{\varnothing} y,y)), \\
  f_1(x_0,x_1,y) &\coloneq G_{0,0^2}(\mathrm{suc}_{0}(x_1 \wedge_{0} f_0(x_0,y),f_0(x_0,y))), \\
  f_2(x_0,x_1,x_2,y) &\coloneq G_{0^2,0^3}(\mathrm{suc}_{0^2}(x_2 \wedge_{0^2} f_1(x_0,x_1,y),f_1(x_0,x_1,y))), \\
  & \vdots  \\
  f_{k-2}(x_0,\dots,x_{k-2},y) &\coloneq G_{0^{k-2},0^{k-1}}(\mathrm{suc}_{0^{k-2}}(x_{k-2} \wedge_{0^{k-2}} f_{k-3}(x_0,\dots,x_{k-3},y),f_{k-3}(x_0,\dots,x_{k-3},y))), \\
  \varphi(x_0,\dots,x_{k-1},y) &\coloneq (f_{k-2}(x_0,\dots,x_{k-2},y) = x_{k-1}).
\end{align*}
A direct calculation now shows that for any $\beta < \alpha$, we have
\begin{align*}
  f_0(a_0,b_{0,\beta}) &= G_{\varnothing,0}(\mathrm{suc}_{\varnothing}(a_0\wedge_{\varnothing} b_{0,\beta},b_{0,\beta})) =  G_{\varnothing,0}(c_{0,\beta}) = b_{1,\beta},\\
  f_1(a_0,a_1,b_{0,\beta}) &= G_{0,0^2}(\mathrm{suc}_{0}(a_1\wedge_{0} b_{1,\beta},b_{1,\beta})) =  G_{0,0^2}(c_{1,\beta}) = b_{2,\beta}, \\
                       & \vdots \\
  f_{k-3}(a_0,\dots,a_{k-3},b_{0,\beta}) &= G_{0^{k-3},0^{k-2}}(\mathrm{suc}_{0^{k-3}}(a_{k-3}\wedge_{0^{k-3}} b_{k-3,\beta},b_{k-3,\beta})) =  G_{0^{k-3},0^{k-2}}(c_{k-3,\beta}) = b_{k-2,\beta}, \\
  f_{k-2}(a_0,\dots,a_{k-2},b_{0,\beta}) & = \dots =  G_{0^{k-2},0^{k-1}}(c_{k-2,\beta}) = \begin{cases}
    a_{k-1}, & \beta~\text{even} \\
    a'_{k-1}, & \beta~\text{odd}
  \end{cases}, \\ %
\end{align*}
and so in particular $\varphi(a_0,\dots,a_{k-1},b_{0,\beta})$ is true if and only if $\beta$ is even.

For any $\beta < \alpha$, let $M_\beta$ be the substructure of $M$ consisting of all elements (in each $P_\eta$) of height less than $\omega\cdot \beta$.

\begin{lem}\label{lem:generation}
  For any $i < k$, $\beta < \alpha$, and $\gamma < \alpha$ with $\beta \leq \gamma$, the substructure of $M$ generated by $M_\beta \cup \{a_0,\dots,\hat{a}_i,\dots,a_{k-1},b_{0,\gamma}\}$ consists of
  \begin{itemize}
  \item the elements of $M_\beta$,
  \item the elements $a_0,\dots,\hat{a}_i,\dots,a_{k-1}$,
  \item for each $j< i$, $0^{\omega\cdot \gamma}\concat 0^n$ and $0^{\omega\cdot \gamma}\concat 1\concat 0^n$ for $n<\omega$ as elements of $P_{0^j}$ (including in particular $c_{j,\gamma}$), and
  \item the element $b_{j,\gamma} \in P_{0^j}$ for each $j \leq i$.
  \end{itemize}
\end{lem}
\begin{proof}
  It is straightforward to check that each of the elements listed in the lemma is generated by the elements of $M_\beta \cup \{a_0,\dots,\hat{a}_i,\dots,a_{k-1},b_{0,\gamma}\}$ and moreover that this list is closed under the functions of $\Lc_{2^{<\omega}}$.
\end{proof}

\begin{lem}\label{lem:same-type}
  For any $i < k$, $\beta < \alpha$, and $\gamma,\gamma'$ with $\beta \leq \gamma,\gamma'$, the map consisting of the identity on $M \cup \{a_0,\dots,\hat{a}_i,\dots,a_{k-1}\}$ and $b_{0,\gamma}\mapsto b_{0,\gamma'}$ extends uniquely to an isomorphism of the substructure of $M$ generated by $M_\beta\cup\{a_0,\dots,\hat{a}_i,\dots,a_{k-1},b_{0,\gamma}\}$ to that generated by $M_\beta\cup\{a_0,\dots,\hat{a}_i,\dots,a_{k-1},b_{0,\gamma'}\}$.

  In particular, we have that in an ambient model of $T_{2^{<\omega}}$, $b_{0,\gamma} \equiv_{M_\beta a_0\dots \hat{a}_i \dots a_{k-1}}b_{0,\gamma'}$.
\end{lem}
\begin{proof}
  The fact that the map extends to an isomorphism is immediate from \cref{lem:generation}. The final statement follows from quantifier elimination for $T_{2^{<\omega}}$.
\end{proof}

Applying \cref{lem:same-type} to the specific case of $\beta$ and $\beta+1$ for even $\beta$ now gives that $\tp(a_0;\dots;a_{k-1}/b_{\leq \beta+1})$ $k$-splits over $b_{0,<\beta}$. Therefore $\tp(a_0;\dots;a_{k-1}/b_{< \alpha})$ $k$-splits over $b_{<\gamma}$ for any $\gamma < \alpha$. Since we can do this for any $k < \omega$ and ordinal $\alpha$, we have that $T_{2^{<\omega}}$ has unbounded $k$-splitting for every $k$.

\bibliographystyle{plain}
\bibliography{../ref}

\begin{thebibliography}{10}

\bibitem{FOP}
A.~Abd~Aldaim, G.~Conant, and C.~Terry.
\newblock Higher arity stability and the functional order property, 2023.

\bibitem{Bailetti2024}
Michele Bailetti.
\newblock A walk on the wild side: Notions of maximality in first-order
  theories, 2024.

\bibitem{Baumgartner1975}
J.~Baumgartner.
\newblock Ineffability properties of cardinals {I}.
\newblock In A.~Hajnal, R.~Rado, and V.~T. S{\'o}s, editors, {\em Infinite and
  Finite Sets}, volume~10 of {\em Colloquia Mathematica Societatis J{\'a}nos
  Bolyai}, pages 109--130. North-Holland, Amsterdam, 1975.

\bibitem{ChernikovUpcoming2025}
Artem Chernikov.
\newblock Forthcoming.

\bibitem{ChernikovOberwolfach2023}
Artem Chernikov.
\newblock Towards higher classification theory.
\newblock In {\em Model Theory: Combinatorics, Groups, Valued Fields and
  Neostability}, volume~2 of {\em Oberwolfach Workshop Reports}, pages
  129--134. Mathematisches Forschungsinstitut Oberwolfach, 2023.
\newblock Abstract in Oberwolfach Report No. 2/2023.

\bibitem{Chernikov_2019}
Artem Chernikov, Daniel Palacin, and Kota Takeuchi.
\newblock On n-dependence.
\newblock {\em Notre Dame Journal of Formal Logic}, 60(2), May 2019.

\bibitem{Friedman_2001}
Harvey~M. Friedman.
\newblock Subtle cardinals and linear orderings.
\newblock {\em Annals of Pure and Applied Logic}, 107(1–3):1–34, January
  2001.

\bibitem{BoundedUltra}
James~E. Hanson.
\newblock Bounded ultraimaginary independence and its total {M}orley sequences.
\newblock {\em Model Theory}, 3(1):39–69, April 2024.

\bibitem{HRUSHOVSKI_2012}
Ehud Hrushovski.
\newblock Groupoids, imaginaries and internal covers.
\newblock {\em Turkish Journal of Mathematics}, January 2012.

\bibitem{Kaplan_2024}
Itay Kaplan, Nicholas Ramsey, and Pierre Simon.
\newblock Generic stability independence and treeless theories.
\newblock {\em Forum of Mathematics, Sigma}, 12, 2024.

\bibitem{Kaplan_2014}
Itay Kaplan and Saharon Shelah.
\newblock A dependent theory with few indiscernibles.
\newblock {\em Israel Journal of Mathematics}, 202(1):59–103, June 2014.

\bibitem{Kim_2008}
Byunghan Kim, Alexei~S. Kolesnikov, and Akito Tsuboi.
\newblock Generalized amalgamation and n-simplicity.
\newblock {\em Annals of Pure and Applied Logic}, 155(2):97–114, September
  2008.

\bibitem{Shelah1986-mn}
Saharon Shelah.
\newblock {\em Around classification theory of models}.
\newblock Lecture notes in mathematics. Springer, Berlin, Germany, 1986
  edition, April 1986.

\bibitem{shelah1990classification}
Saharon Shelah.
\newblock {\em Classification theory and the number of non-isomorphic models}.
\newblock North-Holland Distributors for the U.S.A. and Canada, Elsevier
  Science Pub. Co, Amsterdam New York New York, N.Y., U.S.A, 1990.

\bibitem{Shelah2012}
Saharon Shelah.
\newblock Dependent dreams: recounting types, 2012.

\bibitem{Shelah_2014}
Saharon Shelah.
\newblock Strongly dependent theories.
\newblock {\em Israel Journal of Mathematics}, 204(1):1–83, July 2014.

\bibitem{Takeuchi2OrderProperty}
Kota Takeuchi.
\newblock On 2-order property.
\newblock {\em The Bulletin of Symbolic Logic}, (4):467--483.
\newblock Conference presentation abstract.

\bibitem{tent_ziegler_2012}
Katrin Tent and Martin Ziegler.
\newblock {\em A Course in Model Theory}.
\newblock Lecture Notes in Logic. Cambridge University Press, 2012.

\bibitem{TerryWolf2021}
C.~Terry and J.~Wolf.
\newblock Higher-order generalizations of stability and arithmetic regularity,
  2021.

\bibitem{Terry_Wolf_Irregular}
C.~Terry and J.~Wolf.
\newblock Irregular triads in 3-uniform hypergraphs, 2021.

\bibitem{Vasey_2017}
Sebastien Vasey.
\newblock Indiscernible extraction and {M}orley sequences.
\newblock {\em Notre Dame Journal of Formal Logic}, 58(1), January 2017.

\bibitem{WALKER_2022}
Roland Walker.
\newblock Distality rank.
\newblock {\em The Journal of Symbolic Logic}, 88(2):704–737, August 2022.

\end{thebibliography}

\end{document}